\documentclass[11pt]{amsart}

\usepackage{amsmath}
\usepackage{amssymb}
\usepackage[all]{xy}

\makeatletter
\@addtoreset{equation}{section}
\makeatother

\newtheorem{theorem}{Theorem}[section]
\newtheorem{definition}[theorem]{Definition}

\newtheorem{lemma}[theorem]{Lemma}
\newtheorem{remark}[theorem]{Remark}
\newtheorem{proposition}[theorem]{Proposition}

\newtheorem{cor}[theorem]{Corollary}
\newtheorem{corollary}[theorem]{Corollary}
\newtheorem{example}[theorem]{Example}
\newtheorem{question}[theorem]{Question}
\newtheorem{conjecture}[theorem]{Conjecture}

\def\Cee{{\mathbb C}}
\def\En{{\mathbb N}}

\def\Ree{{\mathbb R}}
\def\Tee{{\mathbb T}}
\def\Zee{{\mathbb Z}}

\def\fH{{\mathcal H}}

\def\fL{{\mathcal L}}

\def\fU{{\mathcal U}}

\def\alp{\alpha}
\def\del{\delta}

\def\eps{\varepsilon}

\def\lam{\lambda}

\def\ome{\omega}

\def\sig{\sigma}

\def\vphi{\varphi}

\def\supp{\mathrm{supp}}

\def\comp{\raisebox{.2ex}{${\scriptstyle\circ}$}}

\def\cross{\negmedspace\times\negmedspace}

\def\setdif{\setminus}
\def\to{\rightarrow}

\def\dpair#1#2{\left\langle #1, #2 \right\rangle}
\def\Im{\textrm{Im}\,}

\def\norm#1{\left\|{#1}\right\|}
\def\Re{\textrm{Re}\,}
\def\til#1{\tilde{#1}}

\def\wbar#1{\overline{#1}}
\def\what#1{\widehat{#1}}

\def\aand{\text{ and }}

\def\ffor{\text{ for }}

\def\iif{\text{ if }}
\def\iin{\text{ in }}

\def\wwhere{\text{ where }}

\def\spn{\mathrm{span}}

\def\re{\mathrm{Re}}

\def\endpf{{\hfill$\square$\medskip}}

\def\a{\mathfrak a}
\def\k{\mathfrak k}
\def\u{\mathfrak u}
\def\z{\mathfrak z}
\def\C{\mathbb C}
\def\R{\mathbb R}
\def\I{\mathbb I}
\def\N{\mathbb N}
\def\al{\alpha}

\def\be{\beta}
\def\DE{\Delta}
\def\de{\delta}
\def\rh{\rho}
\def\et{\eta}

\def\ga{\gamma}
\def\GA{\Gamma}
\def\ve{\varepsilon}
\def\LA{\Lambda}
\def\la{\lambda}
\def\OM{\Omega}
\def\om{\omega}
\def\va{\varphi}

\def\ta{\tau}
\def\sp#1#2{\langle{#1},{#2}\rangle}
\def\cc{\mathfrak{c}}
\def\c{\mathfrak{c}}

\def\g{\mathfrak{g}}
\def\gg{\mathfrak{g}}
\def\a{\mathfrak{a}}
\def\b{\mathfrak{b}}
\def\h{\mathfrak{h}}
\def\k{\mathfrak{k}}
\def\q{\mathfrak{q}}
\def\p{\mathfrak{p}}
\def\n{\mathfrak{n}}
\def\m{\mathfrak{m}}

\def\j{\mathfrak{j}}
\def\s{\mathfrak{s}}
\def\t{\mathfrak{t}}
\def\z{\mathfrak{z}}
\def\u{\mathfrak{u}}
\def\r{\mathfrak{r}}

\def\ga{\gamma}
\def\la{\lambda}
\def\ve{\varepsilon}
\def\si{\sigma}
\def\om{\omega}
\def\ep{\epsilon}
\def\vp{\varphi}
\def\ph{\phi}
\def\ch{\chi}
\def\ta{\tau}
\def\ps{\psi}

\def\C{\mathbb{C}}
\def\ad{\rm{\, ad\, }}
\def\Ad{\rm{\, Ad \,}}
\def\Om{\Omega}

\def\PH{\Phi}
\def\ol#1{\overline{#1}}
\def\nn{\nonumber}

\def\Re{{\mathbb R}}
\def\R{{\mathbb R}}
\def\C{{\mathbb C}}
\def\N{{\mathbb N}}
\def\Q{{\mathbb Q}}
\def\Z{{\mathbb Z}}
\def\T{{\mathbb T}}
\def\I{{\mathbb I}}
\def\wh#1{\widehat{#1}}

  \def\Id{{\mathbb I}}
\def\A{{\mathcal A}}
\def\B{{\mathcal B}}
\def\D{{\mathcal D}}
\def\F{{\mathcal F}}
\def\E{{\mathcal E}}
\def\H{{\mathcal H}}
\def\K{{\mathcal K}}

\def\M{{\mathcal M}}

\def\RR{{\mathcal R}}

\def\CC{{\mathcal C}}
\def\S{{\mathcal S}}

\def\X{{\mathcal X}}
\def\U{{\mathcal U}}
\def\V{{\mathcal V}}
\def\W{{\mathcal W}}

\def\ad{{\text ad}}
\def\tr{{\mathrm Tr}}

\def\iy{\infty}

\def\noin{\noindent}

\def\ol#1{\overline{#1}}
\def\ul#1{\underline{#1}}

\def\cds#1#2{#1,\cdots,#2}

\def\hb#1{\hbox{#1}}
\def\val#1{\vert #1\vert}

\def\no#1{\Vert #1\Vert }
\def\noop#1{\Vert #1\Vert_{\rm op}}
\def\noin#1{\Vert #1\Vert_{\text{ind}\rh}}

\def\pa#1{\{#1\}}

\def\ind#1#2{\hb{ind}_{#1}^{#2}}
\def\CC#1{ C_c(#1)}

\def\supp#1{\text{supp}(#1)}

\def\res#1{_{\vert #1}}
\def\inv{^{-1}}

\def\es{\emptyset}

\def\vs #1#2#3{\vskip #1#2#3 cm{\n}}
\def\hs #1#2#3{\hskip #1#2#3 cm}

\def\me{\medskip\noindent}
\def\hb #1{\hbox{#1}}

\def\un#1{\underline{#1}}

\def\cds{\cdots}

\def\hb#1{\hbox{#1}}
\def\val#1{\vert #1\vert}

\def\pa#1{\{#1\}}

\def\im#1{\hb{\mathrm{im}}(#1)}

\def\sp#1#2{\langle #1,#2\rangle }

\def\ca#1{{\mathcal #1}}

\def\Log#1{\rm{Log}(#1)}

\def\ele{\'el\'ement}
\def\HS{{\mathcal H\mathcal S}}

\def\L#1#2{L^#1(\R^{#2})}
\def\l#1#2{L^{#1}{(#2)}}

\def\Im{\mathrm{\, Im \,}}

\def\stacksx{\tilde\xi}
\def\stacksf{\stackrel{\sim}{f}}
\def\xiphi{\xi\otimes^p_{\rho}\phi}
\def\etapsi{\eta\otimes^{p'}_{\rho}\psi}
\def\pl{\,\parallel\;}
\def\lef({\left(}
\def\rig){\right)}
\def\lan{\langle}
\def\ran{\rangle}

\def\falg{A(G)}
\def\falomeg{A_\ome(G)}
\def\falome#1{A_\ome(#1)}
\def\falbg#1{A_{#1}(G)}  
\def\falomegn{A_{\ome_{G/N}}(G/N)}
\def\falomehh{A_{\ome_H}(H)}
\def\trig{\mathrm{Trig}(G)}
\def\tripig{\mathrm{Trig}_\pi(G)}
\def\etrig{\mathrm{Trig}(\!(G)\!)}  
\def\Tr{\mathrm{Tr}}

\def\fomenorm#1{\left\|#1\right\|_{A_\ome}}
\def\dfomehnorm#1{\left\|#1\right\|_{A_{\ome_H}^*}}
\def\dfomenorm#1{\left\|#1\right\|_{A_\ome^*}}

\def\opnorm#1{\left\|#1\right\|_{\mathrm{op}}}

\newtheorem{fomeisalg}{Proposition}[section]

\def\Cee{{\mathbb C}}
\def\En{{\mathbb N}}

\def\Ree{{\mathbb R}}
\def\Tee{{\mathbb T}}
\def\Zee{{\mathbb Z}}

\def\fH{{\mathcal H}}

\def\fL{{\mathcal L}}

\def\fU{{\mathcal U}}

\def\gg{{\mathfrak g}}
\def\gh{{\mathfrak h}}

\def\gu{{\mathfrak u}}

\def\alp{\alpha}
\def\del{\delta}

\def\eps{\varepsilon}

\def\lam{\lambda}

\def\ome{\omega}

\def\sig{\sigma}

\def\vphi{\varphi}

\def\supp{\mathrm{supp}}

\def\comp{\raisebox{.2ex}{${\scriptstyle\circ}$}}

\def\cross{\negmedspace\times\negmedspace}

\def\setdif{\setminus}
\def\to{\rightarrow}

\def\dpair#1#2{\left\langle #1, #2 \right\rangle}
\def\Im{\textrm{Im}\,}

\def\norm#1{\left\|{#1}\right\|}
\def\Re{\textrm{Re}\,}
\def\til#1{\tilde{#1}}

\def\wbar#1{\overline{#1}}
\def\what#1{\widehat{#1}}

\def\aand{\text{ and }}

\def\ffor{\text{ for }}

\def\iif{\text{ if }}
\def\iin{\text{ in }}

\def\wwhere{\text{ where }}

\def\spn{\mathrm{span}}

\def\falg{A(G)}

\def\trig{\mathrm{Trig}(G)}
\def\trih{\mathrm{Trig}(H)}
\def\tri#1{\mathrm{Trig}(#1)} 
\def\tripig{\mathrm{Trig}_\pi(G)}
\def\etrig{\mathrm{Trig}(\!(G)\!)}  
\def\Tr{\mathrm{Tr}}
\def\t{\mathrm{t}}
\def\unitary#1{\mathrm{U}(#1)}

\title[Beurling-Fourier algebras]
{Beurling-Fourier algebras on compact groups:  spectral theory
}

\author{Jean Ludwig, Nico Spronk and  Lyudmila Turowska}

\begin{document}

\begin{abstract}
For a compact group $G$ we define the Beurling-Fourier algebra $A_\omega(G)$ on $G$ for weights $\omega:\what G\to{\mathbb R}^{>0}$. The classical Fourier algebra corresponds to the case $\omega$ is the constant weight $1$. We study the Gelfand spectrum of the algebra realizing it as a subset of the complexification $G_{\mathbb C}$ defined by McKennon and Cartwright and McMullen. In many cases, such as for polynomial weights, the spectrum is simply $G$.  We discuss the questions when the algebra $A_\omega(G)$ is symmetric and regular. We also obtain various results concerning spectral synthesis for $A_\omega(G)$.  
\end{abstract}

\maketitle

\footnote{


Research of the second named author supported by NSERC Grant 312515-05.
Research of the third named author supported by the Swedish Research Council and the NordForsk Research Network “Operator Algebras and Dynamics” grant 11580.}

\section{Introduction} 
Let $G$ be a compact abelian group with discrete dual group $\what{G}$.  A weight is a function $\ome:\what{G}\to\Ree^{>0}$ for which $\ome(\sig\tau)\leq\ome(\sig)\ome(\tau)$ for $\sig,\tau\iin \what{G}$.  Given such an $\ome$, the {\it Beurling algebra} on $\what{G}$ is given by
\[
\ell^1_\ome(\what{G})=\left\{\bigl(f(\sig)\bigr)_{\sig\in\what{G}}\subset\Cee:\sum_{\sig\in\what{G}}|f(\sig)|\ome(\sig)<\infty\right\}
\]
and is easily verified to be a commutative Banach algebra under convolution.  We say $\ome$ is bounded if $\inf_{\sig\in\what{G}}\ome(\sig)>0$.  In this case $\ell^1_\ome(\what{G})$ is a subalgebra of the group algebra $\ell^1(\what{G})$.  In particular we can apply the Fourier transform to obtain an algebra $A_\ome(G)$ of continuous functions on $G$.  Beurling algebras been studied by several people, e.g. Domar \cite{domar}, Reiter \cite{reiter}.  If $\omega\equiv 1$ we get the classical Fourier algebra $A(G)$, i.e.\ the space of transforms of $\ell^1(G)$.

For any locally compact group $G$ the Fourier algebra was defined by Eymard \cite{eymard} as the algebra of matrix coefficients of the left regular representation.  In the case that $G$ is compact it is well-known that $A(G)$ can be identified with the space of operator fields indexed over the set irreducible representations:
\[
\left\{\bigl(f(\pi)\bigr)_{\pi\in\what{G}}\in\prod_{\pi\in\what{G}}\fL(\fH_\pi):\sum_{\pi\in\what{G}}d_\pi\norm{f(\pi)}_1<\infty\right\}.
\]
where $\fL(\fH_\pi)$ is the space of linear operators on the Hilbertian representation space $\fH_\pi$, and $\norm{\cdot}_1$ is the trace norm.  In light of the definition of Beurling algebras, above, it is natural to define a weight on $\what{G}$ as a function $\ome:\what{G}\to\Ree^{>0}$ which satisfies $\ome(\sig)\leq\ome(\pi)\ome(\pi')$, whenver $\sig$ may be realised as a subrepresentation of $\pi\otimes\pi'$.  Thus it is natural to define the {\it Beurling-Fourier algebra} $A_\ome(G)$ so it may be identified with the space of operator fields
\[
\left\{\bigl(f(\pi)\bigr)_{\pi\in\what{G}}\in\prod_{\pi\in\what{G}}\fL(\fH_\pi):\sum_{\pi\in\what{G}}\norm{f(\pi)}_1d_\pi\ome(\pi)<\infty\right\}.
\]
We show that this definition always provides a semi-simple commutative Banach algebra; moreover, when $\ome$ is bounded --- i.e.\ $\inf_{\pi\in\what{G}}\ome(\pi)>0$ --- this is a subalgebra of the Fourier algebra.

To describe the spectrum of $A_\omega (G)$ we require an abstract Lie theory which is built from Krein-Tanaka duality and was formalised separately by McKennon \cite{mckennon} and Cartwright and McMullen \cite{cartwrightm} in the 70s. This Lie theory allowed to develop the complexification $G_{\mathbb C}$ even for non-Lie groups $G$. The Gelfand spectrum of $A_\omega(G)$ is shown to be a subset of $G_{\mathbb C}$.
In contrast to the Fourier algebra $A(G)$ for which the spectrum is $G$, $A_\omega(G)$ can have a larger spectrum $G_\omega$. Examples of such weights and groups $G$ are given in Section~\ref{spectrum_section}.
We explore conditions for which $G_\omega=G$.
In Section~\ref{spectrum_section} we prove that for symmetric weights $\omega$ the equality holds if and only if the algebra $A_\omega(G)$ is symmetric. 
In Section~\ref{spectrum_section} we define the notion of exponential growth for a weight $\omega$ and showed that $G_\omega=G$ if $\omega$ is of non-exponential type. 
Examples of weights of non-exponential growth are polynomial ones defined in Section~\ref{polynomial}. For such weights we could introduce a smooth functional calculus and use this to show that $A_\omega(G)$ is a regular algebra. This gives us a possibility to study the property of spectral synthesis. Adapting arguments from \cite{Lu-Tu} on the Fourier algebra of compact Lie groups we  prove that if $E$ is a compact subset of a Lie group $G$, then $E$ is a set of weak synthesis if it is of smooth synthesis. Moreover we give an estimate of the corresponding nilpotency degree in terms of the degree of the polynomial weight $\omega$.
As a consequence we obtain conditions for a one-point set to be a set of spectral synthesis for $A_\omega(G)$. Finally in the last subsection we study a connection between spectral synthesis and operator synthesis in the spirit of \cite{Ni-Lu}.

We note that Lee and Samei in \cite{lee-samei} suggest a more general approach to the notion of weight on $\widehat G$. Their central weights for compact groups turn out to coincide with our notion of weight. However in their paper they mainly study the properties of operator amenability and Arens regularity.

\section{An abstract Lie theory for compact groups}\label{sec:abslietheory}

In this section we remark on some consequences of the Krein-Tannaka duality
theory for compact groups which allow us to define a ``complexification'' $G_\Cee$
for any compact group $G$.  This object will be necessary for us to develop
a description of the spectrum of general Beurling-Fourier algebras.
The theory in this section was thoughourghly developed by McKennon
\cite{mckennon} and Cartwright and McMullen \cite{cartwrightm}.  We shall be requiring it
to an extent that a summary is warranted.  

Let $G$ be a compact group with dual object $\what{G}$, which, by mild
abuse of notation, we treat as
a set of unitary irreducible representations: $\pi:G\to\fU(\fH_\pi)$. We let
$d_\pi=\dim{\fH_\pi}$.  For representations
$\sig,\pi$ of $G$, we will use the notation $\sig\subset\pi$ to denote that
$\sig$ is unitarily equivalent to a subrepresentation of $\pi$.

We let $\trig$ denote the space of trigonometric polynomials, i.e.\ the
span of matrix coefficients of elements of $\what{G}$, which is well-known
to be an algebra of functions under pointwise operations.
We note that
$\trig=\bigoplus_{\pi\in\what{G}}\tripig$, where $\tripig$ is the span
of matrix coefficients of $\pi$.  For $u\iin\trig$, we let
\[
\hat{u}(\pi)=\int_G u(s)\pi(s^{-1})ds
\]
which may be understood to be an element
of the space of linear operators $\fL(\fH_\pi)$.  We caution the reader that
our notation differs from that in \cite[(28.34)]{hewittrII}.
We shall make an identification
between two linear dual space $\trig^\dagger$ and the product
$\prod_{\pi\in\what{G}}\fL(\fH_\pi)$ via
\begin{equation}\label{eq:trigdual1}
\dpair{u}{(T_\pi)_{\pi\in\what{G}}}
=\sum_{\pi\in\what{G}}d_\pi\Tr(\hat{u}(\pi)T_\pi).
\end{equation}
It follows from the orthogonality relations between matrix coefficients that for a matrix
coefficient
$\pi_{\xi,\eta}(s)=(\pi(s)\eta,\xi)$
$$\langle \pi_{\xi,\eta},(T_\pi)_{\pi\in \wh G}\rangle=(T_\pi\eta,\xi).$$
In the notation of (\ref{eq:trigdual1}), we will write for any $T\iin\trig^\dagger$
and $\pi\in\wh{G}$
\begin{equation}\label{eq:trigdual2}
\pi(T)=T_\pi\iin\fL(\fH_\pi).
\end{equation}
In particular, we identify $G$, qua evaluation functionals on $\trig$, with
$\{(\pi(s))_{\pi\in\what{G}}:s\in G\}$.  Thus we gain the Fourier inversion
\begin{equation}\label{eq:trigdualG}
\dpair{u}{s}=\sum_{\pi\in\what{G}}d_\pi\Tr\bigl(\hat{u}(\pi)\pi(s)\bigr)=u(s).
\end{equation}
Moreover, by \cite[(30.5)]{hewittrII}, for example
\begin{equation}\label{eq:unitarypart}
G\simeq\{T\in\trig^\dagger:
\pi(T)\in\fU(\fH_\pi)\text{ for each }\pi\iin\what{G}\}.
\end{equation}

We note that $\trig^\dagger\simeq \prod_{\pi\in\what{G}}\fL(\fH_\pi)$ has an obvious
product and involution which respects the formulas $\pi(TT')=\pi(T)\pi(T')$ and
$\pi(T^*)=\pi(T)^*$ for $\pi\iin\what{G}$ and $T,T'\iin\trig^\dagger$.  Moreover
the action of $\trig^\dagger$ on $\trig$ given by $\what{T\cdot u}(\pi)
=\pi(T)\hat{u}(\pi)$ for $\pi\in\what{G}$ satisfies 
$\dpair{T'}{T\cdot u}=\dpair{T'T}{u}$.  
We note that the involution satisfies
$\dpair{T^*}{u}=\overline{\dpair{T}{u^*}}$, where $u^*(s)=\wbar{u(s^{-1})}$.  

With the notation above we define
\begin{align}
G_\Cee&=\left\{\theta\in\trig^\dagger:\dpair{\theta}{uu'}=
\dpair{\theta}{u}\dpair{\theta}{u'}\ffor u,u'\iin\trig\right\} \notag \\
&=\left\{\theta\in\trig^\dagger:\theta\cdot(uu')=(\theta\cdot u)(\theta\cdot u')
\ffor u,u'\iin\trig\right\} \label{eq:GCmodule} \\
\gg_\Cee&=\left\{X\iin\trig^\dagger:
\begin{matrix} \dpair{X}{uu'}=\dpair{X}{u}u'(e)+u(e)\dpair{X}{u'} \\
\text{for all }u,u\iin\trig\end{matrix}\right\} \notag \\
&=\left\{X\iin\trig^\dagger:
\begin{matrix}X\cdot(uu')=(X\cdot u)u'+u(X\cdot u') \label{eq:gLiemodule} \\
\text{for all }u,u\iin\trig\end{matrix}\right\} \\
\aand  \gg&=\{X\in\gg_\Cee:X^*=-X\} \notag
\end{align}
where the equivalent descriptions (\ref{eq:GCmodule}) and (\ref{eq:gLiemodule}),
of $G_\Cee$ and $\gg_\Cee$ respectively, can be checked by straightforward calculation.
We observe that it is immediate from (\ref{eq:GCmodule}) that $G_\Cee$
is closed under the product in $\trig^\dagger$.  It is a standard fact, see
\cite[(30.26)]{hewittrII} for example, that $G_\Cee$ is closed under inversion and hence a group.
Moreover, from
(\ref{eq:gLiemodule}) it is immediate that $\gg_\Cee$ is a complex Lie algebra
under the usual associative Lie bracket: $[X,X']=XX'-X'X$.  In particular
$\gg$ is a real Lie subalgebra.  It is obvious that
$\trig^\dagger\simeq\prod_{\pi\in\what{G}}\fL(\fH_\pi)$ is closed under analytic 
functional calculus.  Thus by standard calculation 
we find that for $X\iin\trig^\dagger$,
$X\in\gg_\Cee$ if and only if $\exp(tX)\in G_\Cee$ for each $t\iin\Ree$;
see \cite[Prop.\ 3]{cartwrightm} (or see comment after (\ref{eq:lielike}))
for one direction, and differentiate 
$t\mapsto\exp(tX)\cdot (uu')$ to see the other.  
Moreover, by further employing (\ref{eq:unitarypart}), we see for $X\iin\trig^\dagger$ that
$X\in\gg$ if and only if $\exp(tX)\in G$ for each $t\iin\Ree$.

We record some of the basic properties of the group $G_\Cee$ and the
Lie algebras $\gg_\Cee$ and $\gg$.

\begin{proposition}\label{prop:GCproperties}
{\bf (i)} $G_\Cee$ admits polar decomposition:  each $\theta\iin G_\Cee$ can be
written uniquely as $\theta=s|\theta|$, i.e.\ $\pi(\theta)=\pi(s)|\pi(\theta)|$ for each
$\pi\iin\what{G}$.  Hence, each such $|\theta|$ is an element of $G_\Cee$.

{\bf (ii)} If $\theta\in G_\Cee$, then $\theta^*\in G_\Cee$ too.
If $\theta\in G_\Cee^+=\{T\in G_\Cee:\pi(T)\geq 0\text{ for }\pi\in\what{G}\}$, 
then for each $z\in \Cee$,
$\theta^z\in G_\Cee$ too; moreover if $t\in\Ree^{\geq 0}$, then $\theta^t\in G_\Cee^+$.

{\bf (iii)} If the connected component of the identity $G_e$ is a Lie group, then
$\gg$ is isomorphic to the usual Lie algebra of $G$ and $\gg_\Cee$ is its complexification.

{\bf (iv)} We have $\exp(\gg)\subset G_e$ and is dense, and $\exp(i\gg)=G_\Cee^+$.
The map $(s,X)\mapsto s\exp(iX):G\times\gg\to G_\Cee$ is a homeomorphism, where
$\gg$ and $G_\Cee$ have relativised topologies as subsets of $\trig^\dagger$
whose topology is the weak topology induced by (\ref{eq:trigdual1}).
\end{proposition}

\proof Part (i) and the second part of (ii) can be found in
\cite[Cor.\ 1 \& Thm.\ 2]{mckennon}.  We note that the proof can be conceptualised
a bit differently, the ideas of which we sketch below.  

It is well known that $\what{G\times G}$ is given by
Kronecker products $\{\pi\times\pi':\pi,\pi'\in\what{G}\}$, and hence
$\tri{G\cross G}\simeq\trig\otimes\trig$.  We let $m:\tri{G\cross G}\to\trig$
denote pointwise multiplication and 
\[
m^\dagger:\trig^\dagger\to\tri{G\cross G}^\dagger
\simeq \prod_{\pi,\pi'\in\what{G}}\fL(\fH_\pi\otimes\fH_{\pi'})
\]
denote its adjoint map.  For each $\pi,\pi'\iin\what{G}$ and each
$\sig\iin\what{G}$ 
there are $m(\sig,\pi\otimes\pi')$ (this number may be $0$) partial isometries 
$V_{\sig,i}:\fH_\sig\to\fH_\pi\otimes \fH_\pi$ with pairwise disjoint ranges, for which
\begin{equation}\label{eq:tensdec}
\pi\otimes\pi'
=\sum_{\sig\in\what{G}}\sum_{i=1}^{m(\sig,\pi\otimes\pi')}V_{\sig,i}\sig(\cdot)V_{\sig,i}^*
\end{equation}
where we adopt the convention that an empty sum is $0$.
Then we may calculate for $T\iin\trig^\dagger$ that
\begin{equation}\label{eq:coproduct}
\pi\times\pi'(m^\dagger T)=
\sum_{\sig\in\what{G}}\sum_{i=1}^{m(\sig,\pi\otimes\pi')}V_{\sig,i}\sig(T)V_{\sig,i}^*.
\end{equation}
It follows readily that $m^\dagger$ is a $*$-homomorphism; in fact it
is the well-known coproduct map, see \cite[Ex.\ 1.2.5]{timmermann}.  In particular,
it is easy to calculate from the definitions of $G_\Cee$ and $\gg_\Cee$ that
\begin{align}
G_\Cee&=\left\{\theta\in\trig^\dagger:\theta\otimes\theta
=m^\dagger(\theta)\right\} 
\label{eq:grouplike}\\
\gg_\Cee&=\left\{X\in\trig^\dagger:
X\otimes I+I\otimes X=m^\dagger(X) \right\} \label{eq:lielike}
\end{align}
Where $I$ is the identity element of $\trig^\dagger$.
We note that it is easy to check that exponentiating elements of
$\gg_\Cee$ is (\ref{eq:lielike}) gives elements of the form
(\ref{eq:grouplike}).

Now (i) follows from (\ref{eq:coproduct}), (\ref{eq:grouplike}),
(\ref{eq:unitarypart}) and the uniqueness 
of polar decomposition.  Any multiplicative analytic function 
$\psi:\Ree^{>0}\to\Cee^{\not=0}$,
(respectively, multiplicative anti-analytic function $\psi:\Cee\to\Cee$) 
will thus satisfy 
\[
\psi(\theta)\otimes\psi(\theta)=\psi(\theta\otimes I\, I\otimes\theta)
=\psi(\theta\otimes\theta)=\psi(m^\dagger (\theta))=m^\dagger(\psi(\theta))
\]
for $\theta\in G_\Cee^+$ (respectively, $\theta\in G_\Cee$). 
Hence, again by (\ref{eq:grouplike}), $\psi(\theta)\in G_\Cee$; moreover
$\psi(\theta)\in G_\Cee^+$ if $\psi(\Ree^{>0})\subset\Ree^{>0}$.  Thus
we obtain (ii).

Part (iii) is \cite[Cor.\ 4]{cartwrightm}, while part (iv) is
\cite[Prop.\ 4]{cartwrightm} (see also \cite[Thm.\ 3]{mckennon}.  \endpf

If $H$ is another compact group, and $\sig:G\to H$ is a continuous homomorphism
then $\sig$ induces a $*$-homomorphism $\sig:\trig^\dagger\to\tri{H}^\dagger$.
Indeed, if we assign $\trig^\dagger\simeq\prod_{\pi\in\what{G}}\fL(\fH_\pi)$
the linear topology from the dual pairing (\ref{eq:trigdual1}) then
$\spn G$ is dense in $\trig^\dagger$, since $\spn\pi(G)=\fL(\fH_\pi)$
for each $\pi$ by Schur's lemma.  The map $\sum_{j=1}^n\alp_js_j\mapsto
\sum_{j=1}^n\alp_j\sig(s_j):\spn G\to\spn H$ is well defined, being the relative
adjoint of $u\mapsto u\comp\sig:\trig\to\tri{H}$, and is clearly a $*$-homomorphism.  
Hence this map extends as claimed.  It follows that $\sig|_{G_\Cee}
:G_\Cee\to H_\Cee$ is a homomorphism of groups, while $\sig|_{\gg_\Cee}:\gg_\Cee\to
\gh_\Cee$ is a homomorphism of $\Cee$-Lie algebras, and
$\sig|_{\gg}:\gg\to\gh$ is a homomorphism of $\Ree$-Lie algebras, where
$\gh$ is the Lie algebra of $H$.  Moreover it is immediate that
$\exp(\sig(X))=\sig(\exp(X))$ for $X\iin\gg_\Cee$.  Finally, if
$K$ is a third compact group and $\tau:H\to K$ is a continuous homomorphism,
then $\tau\comp\sig:G\to K$ extends to a $*$-homomorphism
$\tau\comp\sig:\trig^\dagger\to\tri{K}^\dagger$ and restricts
to a group, respectively Lie algebra, homomorphism where appropriate.

If $\pi\in\what{G}$, we  denote  the linear dual space of $\fH_\pi$ by $\fH_{\bar{\pi}}$.
For $A\iin\fL(\fH_\pi)$, let $A^\t\iin\fL(\fH_{\bar{\pi}})$ denote its linear
adjoint.  For $s\in G$ we define  $\bar{\pi}(s)=\pi(s^{-1})^t$. Then $\bar\pi$ is a unitary representation on $\fH_{\bar\pi}$ called the conjugate representation of $\pi$.

\begin{corollary}\label{prop:ggproperties1}
For any $\theta\iin G_\Cee^+$ and $\pi\iin\what{G}$ we have 
$\bar{\pi}(\theta)=\pi(\theta^{-1})^\t$.
\end{corollary}

\begin{proof}  Let $\unitary{\fH_\pi}$ denote the unitary group on $\fH_\pi$ and
$\gu(\fH_\pi)$ its Lie algebra, which we may regard in the classical sense by (iii)
of the proposition, above.  
Let $\gamma:\unitary{\fH_\pi}\to\unitary{\fH_{\bar{\pi}}}$ 
be given by $\gamma(u)=(u^{-1})^\t$. 
The map $U\mapsto -U^\t:\gu(\fH_\pi)\to\gu(\fH_{\bar{\pi}})$
is a Lie algebra homomorphism.  Moreover, if we write $U=iH$ where $H$ is 
hermitian we have
\[
\exp(-U^\t)=\exp(-iH^\t)=(\exp(iH)^{-1})^\t=\gamma(\exp(U)).
\]
We see that $d\gamma(U)=-U^\t$, $U\in \gu(\fH_\pi)$ and hence its unique $\Cee$-linear
extension to the complexification satisfies $d\gamma(X)=-X^\t$ for $X\in \gu(\fH_\pi)_\Cee$.

We shall regard $\pi:G\to\unitary{\fH_\pi}$
as a homomorphism, as above, so $\bar{\pi}=\gamma\comp\pi$.
If $\theta\in G_\Cee^+$, we use (iv) in the proposition above to
write $\theta=\exp(iX)$ where $X\in\gg$ (so $iX$ is hermitian).  Thus
we compute
\[
\bar{\pi}(\theta)=\exp(i\bar{\pi}(X))=\exp(i\gamma\comp\pi(X))
=\exp(-i\pi(X)^\t)=\pi(\theta^{-1})^\t
\]
as desired.  \end{proof}

\section{Beurling-Fourier algebras. Definition.}

Let $ G $ be  a compact group. A {\it weight} on $\what{G}$ is a function 
$\ome:\what{G}\to\Ree^{>0}$
such that
\[
\ome(\sig)\leq\ome(\pi)\ome(\pi')\text{ whenever }
\sig\subset\pi\otimes\pi'.
\]
We say that $\ome$ is {\it bounded} if $\inf_{\pi\in\what{G}}\ome(\pi)>0$, and
{\it symmetric} if $\ome(\bar{\pi})=\ome(\pi)$ for each $\pi\iin\what{G}$.
Note that since $1\subset\pi\otimes\bar\pi$, any symmetric weight is automatically 
bounded with $\ome(\pi)^2
=\ome(\pi)\ome(\bar{\pi})\geq \ome(1)$.  Of course, $\ome(1)\geq 1$ since
$1\otimes\pi=\pi$ for each $\pi$.

We let $\etrig=\prod_{\pi\in\what{G}}\tripig$ denote the space of formal
trigonometric series.  For $u\in\etrig$ and $\pi\iin\what{G}$, the definition 
$\hat{u}(\pi)$ can be regarded as a formal integral.
If $\ome$ is a weight on $\what{G}$, we define
\[
\falomeg=\left\{u\in\etrig:\fomenorm{u}=
\sum_{\pi\in\what{G}}\norm{\hat{u}(\pi)}_1d_\pi\ome(\pi)<\infty\right\},
\]
where, for each $\pi\iin\what{G}$, $d_\pi=\dim{\fH_\pi}$.
It is obvious that $\falomeg$ is
a Banach space with norm $\fomenorm{\cdot}$.  We call $\falomeg$ the
{\it Beurling-Fourier algebra} with weight $\ome$.  We will see that it
is a Banach algebra, below.  Note that if $\ome$ is the constant weight
$1$, then $\falbg{1}=\falg$ is the Fourier algebra of $G$.

\begin{fomeisalg}\label{prop:fomeisalg}
For any weight $\ome$, $\falomeg$ is a Banach algebra under the product 
extending pointwise multiplication on $\trig$.  
If $\ome$ is bounded, then $\falomeg\subset\falg$, and
hence is an algebra of continuous functions on $G$.
\end{fomeisalg}

\begin{proof}  It is clear that $\trig$ is $\fomenorm{\cdot}$-dense in
$\falomeg$, hence it suffices to verify that $\fomenorm{\cdot}$
is an algebra norm on $\trig$.  Let us first consider a pair
of basic coefficients, $u=\dpair{\pi(\cdot)\xi}{\eta}$ and
$u'=\dpair{\pi'(\cdot)\xi'}{\eta'}$, where $\pi,\pi'\in\what{G}$,
$\xi,\eta\in\fH_\pi$, and $\xi',\eta'\in\fH_{\pi'}$.  Note that by
the Schur orthogonality relations
\[
\fomenorm{\dpair{\pi(\cdot)\xi}{\eta}}=\norm{\hat{u}(\pi)}_1d_\pi\ome(\pi)=
\norm{\xi}\norm{\eta}\ome(\pi)
\]
and similar holds for any basic matrix coefficient.  
There exist,
not necessarily distinct, $\sig_1,\dots,\sig_m\iin\what{G}$ such that
$V\pi\otimes\pi'(\cdot)V^*=\bigoplus_{j=1}^m\sig_j$ for some unitary operator $V$, 
and, with those, associated
pairwise orthogonal projections onto the reducing subspaces,
$p_1,\dots,p_m$ on $V(\fH_\pi\otimes\fH_{\pi'})$.  Let for each $i$, $V_i=p_iV$
(so $V_i=V_{\sig_i,j}^*$ in the notation of (\ref{eq:tensdec}). 
Hence we have
\begin{align*}
uu'&=(\pi\otimes\pi'(\cdot)\xi\otimes\xi',\eta\otimes\eta') \\
&=\left(\bigoplus_{j=1}^m\sig_j(\cdot)\sum_{j'=1}^mV_{j'}(\xi\otimes\xi')\;,
\;\sum_{j''=1}^mV_{j''}(\eta\otimes\eta')\right) \\
&=\sum_{j=1}^m (\sig_j(\cdot)V_j(\xi\otimes\xi'),
V_j(\eta\otimes\eta'))
\end{align*}
Hence, since $\ome$ is a weight, we have
\begin{align*}
\fomenorm{uu'}
&\leq\sum_{j=1}^m\fomenorm{(\sig_j(\cdot)V_j(\xi\otimes\xi'),
V_j(\eta\otimes\eta'))} \\
&=\sum_{j=1}^m\norm{V_j(\xi\otimes\xi')}
\norm{V_j(\eta\otimes\eta')}\ome(\sig_j) \\
&\leq\left(\sum_{j=1}^m\norm{V_j(\xi\otimes\xi')}
\norm{V_j(\eta\otimes\eta')}\right)\ome(\pi)\ome(\pi') \\
&\leq\left(\sum_{j=1}^m\norm{V_j(\xi\otimes\xi')}^2\right)^{1/2}
\left(\sum_{j=1}^m\norm{V_j(\eta\otimes\eta')}^2\right)^{1/2}
\ome(\pi)\ome(\pi') \\
&=\norm{\xi\otimes\xi'}\norm{\eta\otimes\eta'}\ome(\pi)\ome(\pi')
=\fomenorm{u}\fomenorm{u'}.
\end{align*}
Since for each $\pi$, each $T\iin\fL(\fH_\pi)$ with
$\norm{T}_1\leq 1$ is in the convex hull of rank one operators,
it follows that the convex hull of
\[
\{(\pi(\cdot)\xi,\eta):\pi\in\what{G},\xi,\eta\in\fH_\pi,
\norm{\xi}\norm{\eta}=1/\ome(\pi)\}
\]
is the $\fomenorm{\cdot}$-unit ball of $\trig$.  Hence it follows
that $\fomenorm{\cdot}$ is sub-\-multi\-pli\-cative on $\trig$.

Now if $\ome$ is bounded, with $C=\inf_{\pi\in\what{G}}\ome(\pi)>0$, then
we have for $u\iin\tripig$ that $ u $ is contained in $ A(G) $ and that
\begin{eqnarray}\label{agawg}
 \nn \no{u}_{A}&=&\sum_{\pi\in\what{G}}\norm{\hat{u}(\pi)}_1d_\pi\\
  &\leq&
\sum_{\pi\in\what{G}}\norm{\hat{u}(\pi)}_1\frac{1}{C}
\om(\pi) d_\pi\\
\nn  &=&\frac{1}{C}\no{u}_{A_\om}.
\end{eqnarray}
Thus $\falomeg$ is contained in $ A(G) $ and is therefore  a space of
continuous functions.
\end{proof}
\begin{example} \rm (1)
 If $\om \equiv 1$ then $A_\om (G)=A(G)$, the Fourier algebra of $G$.

(2) If $\om(\pi)=d_\pi$, $\pi\in\wh G$ (we say that $\om$ is the dimension weight) then
$A_\om(G)=A_\gamma(G)$, an algebra 
studied by B.E. Johnson (\cite{johnson}) and which is the image of the map 
from $A(G)\otimes^\gamma A(G)$ to $A(G)$  given on  elementary tensors by
$f\otimes g\mapsto f\ast\check g$, where $\check g(t)=g(t\inv)$; or, as shown
in \cite{forrestss}, is the image of the map $f\otimes g\mapsto f\ast g$
from $A(G\times G)$ to $A(G)$.

(3) If $G=\Tee^n$ and $\om$ is a weight on $\what{\Tee^n}\simeq\Zee^n$, 
then $A_\om(G)\simeq l^1(\Zee^n,\ome)$ is a Beurling algebra (see \cite{reiter}).

\end{example}

\begin{remark}\label{fourinv}
\rm   For a bounded weight $ \om $ on a compact group $ G $ we can extend the
Fourier inversion formula (\ref{eq:trigdualG}) to elements $ u\in A_\om(G) $,  since then $ A_
\om(G)\subset A(G)$, and series from elements of the latter are summable. 
That is for $ u\in A_ \om(G)$ and $ s\in G $  we have that
 \begin{eqnarray}\label{invers_formula}
 \nn u(s)= \sum_{\pi\in\what{G}}\Tr({\pi(s) \hat{u}(\pi)})d_\pi.
\end{eqnarray}

 \end{remark}

   Let $ \rh $ denote the right  translation on $\trig$, i.e.\ 
for an element $u\iin\trig$ we let\begin{eqnarray}
 \nn \rh(t)u(s):=u(st), s,t\in G,
\end{eqnarray}
and let $\lambda$ be the left translation on $\trig$ defined by
\begin{eqnarray}
 \nn \la(t)u(s):=u(t\inv s), s,t\in G.
\end{eqnarray}

\begin{proposition}\label{transl}
For every weight $ \om $ on $ \widehat G $, 
right and left translations extend to isometries of $ A_\om(G) $.
 \end{proposition}
\begin{proof}
Indeed, for an element $ u\iin \trig $ and $ t\iin G $ we have that
\begin{eqnarray}
 \nn \widehat {(\rh(t)u)}(\pi)= \pi(t)\widehat u(\pi),\ 
\widehat {(\la(t)u)}(\pi)= \widehat u(\pi) \pi(t\inv), \ \pi\in \widehat  G
\end{eqnarray}
and therefore
\begin{eqnarray}
 \nn \no{\rh(t)u}_{A_\om}=\sum_{\pi\in\wh G}\no{\pi(t)\wh u(\pi)}_1d_\pi \om(\pi)
= \sum_{\pi\in\wh G}\no{\wh u(\pi)}_1d_\pi \om(\pi)=\no u_{A_\om}
\end{eqnarray}
and similarly
$$\no{\la(t)u}_{A_\om}=\no u_{A_\om}.$$
 \end{proof}

 Let for a bounded weight $ \om $ on $ \wh{G} $
\begin{eqnarray}
 \nn L^2_\om(G)&=&\left\{\xi\in \etrig:
\no
\xi_{2,\om}^2:=\sum_{\pi\in\wh{G}}\no{\hat{\xi}(\pi)}
_2^2 d_\pi\om(\pi)<\infty\right\},
\end{eqnarray}
where $\no{\hat{\xi}(\pi)}_2$
denotes the Hilbert-Schmidt norm of the operator $\hat{\xi}(\pi) $.
The assumption that $\ome$ is bounded gives us that $L^2_\ome(G)\subset L^2(G)$.

It is well-known that the Fourier algebra $A(G)$ coincides with  the family of functions
$$A(G)=\left\{s\mapsto (f\ast  g)(s)=\int_Gf(t)g( t\inv s)dt:f, g\in L^2(G)\right\}. $$

We observe that the product $\ome_1\ome_2$, of two weights $\ome_1,\ome_2$, is a weight; 
and if $\vphi:\Ree^{>0}\to
\Ree^{>0}$ is a non-decreasing
sub-multiplicative function, then $\vphi\comp\ome$ is a weight for any weight $\ome$.

\begin{proposition}  \label{prop:ltwofactor}
Let $\ome_1,\ome_2$ be bounded weights on $\what{G}$
and $\ome=(\ome_1\ome_2)^{1/2}$.  Then
 $$A_\omega(G)=\left\{s\mapsto (f\ast  g)(s)=\int_Gf(t)g( t\inv s)dt:f\in L_{\om_1}^2(G)
\aand g\in L_{\ome_2}^2(G)\right\}.$$
 \end{proposition}
 \begin{proof}
 Let $u\in A_{\omega}(G)$. For each $\pi\in\wh{G}$ consider the polar decomposition of $\hat u(\pi)$:
 $\hat u(\pi)=V(\pi)|\hat u(\pi)|$. Let $a(\pi)=
\left(\frac{\ome_1(\pi)}{\ome(\pi)}\right)^{1/2}V(\pi)|\hat u(\pi)|^{1/2}$ and $b(\pi)=
\left(\frac{\ome_2(\pi)}{\ome(\pi)}\right)^{1/2}|\hat u(\pi)|^{1/2}$.
 We have
 $$\sum_{\pi\in\wh{G}} d_{\pi}\omega(\pi)\|a(\pi)\|^2_2\leq\sum_{\pi\in\wh{G}} d_{\pi}\omega_1(\pi)\||\hat u(\pi)|^{1/2}\|^2_2=
 \sum_{\pi\in\wh{G}} d_{\pi}\omega_1(\pi)\|\hat u(\pi)\|_1<\infty$$ 
and similarly
 $$\sum_{\pi\in\wh{G}} d_{\pi}\omega(\pi)\| b(\pi)\|^2_2\leq \sum_{\pi\in\wh{G}} d_{\pi}\omega_2(\pi)\|\hat u(\pi)\|_1<\infty.$$
 Thus if $f,g\in L^2(G)$ are such that $\hat g(\pi)=a(\pi)$ and $\hat f(\pi)=b(\pi)$ then
 $f\in L^2_{\omega_1}(G)$, $g\in L^2_{\omega_2}(G)$ and $u=f\ast g$. Hence $A_\omega(G)\subset L^2_{\omega_1}(G)*L^2_{\omega_2}(G)$.

 Take now  $f\in L^2_{\omega_1}(G)$, $g\in L^2_{\omega_2}(G)$ and let $u=f\ast g$. Then
 \begin{align*}
 \sum_{\pi\in\wh{G}}d_{\pi}\omega(\pi)&\|\hat u(\pi)\|_1 =\sum_{\pi\in\wh{G}}d_{\pi}\omega(\pi)\|\hat g(\pi)\hat f(\pi)\|_1\\
 &\leq\sum_{\pi\in\wh{G}}d_{\pi}\omega_1(\pi)^{1/2}\omega_2(\pi)^{1/2}\|\hat g(\pi)\|_2\|\hat f(\pi)\|_2\\
 &\leq\left(\sum_{\pi\in\wh{G}}d_{\pi}\omega_1(\pi)\|\hat g(\pi)\|_2^2\right)^{1/2}\left(\sum_{\pi\in\wh{G}}d_{\pi}\omega_2(\pi)\|\hat f(\pi)\|_2^2\right)^{1/2}\\
 &=\|f\|_{2,\ome_1}\|g\|_{2,\ome_2}
 \end{align*}
 Thus $u\in A_\omega(G)$.
 \end{proof}


\section{The spectrum of a Beurling-Fourier algebra}\label{spectrum_section}

It follows from the identification 
$\trig^\dagger\simeq\prod_{\pi\in\what{G}}\fL(\fH_\pi)$
given in (\ref{eq:trigdual1}) that
for a weight $\ome$, $\falomeg$ has continuous dual space
\begin{equation}\label{eq:trigdual}
\falomeg^*=\left\{T\in\trig^\dagger:\dfomenorm{T}=\sup_{\pi\in\what{G}}
\frac{\opnorm{\pi(T)}}{\ome(\pi)}<\infty\right\}.
\end{equation}
The definition of $G_\Cee$ then immediately gives Gelfand spectrum
\begin{equation}\label{eq:spectrum}
G_\ome=\what{\falomeg}=\left\{\theta\in G_\Cee:\sup_{\pi\in\what{G}}
\frac{\opnorm{\pi(\theta)}}{\ome(\pi)}<\infty\right\}
\end{equation}
where $<\infty$ may be replaced by $\leq 1$, by well-known theorem of Gelfand
that multiplicative functionals on Banach algebras are automatically contractive.
Proposition \ref{prop:GCproperties} provides that $G_\Cee$ is closed under polar 
decomposition, and that for $\theta\iin G_\Cee$ and $\pi\in\what{G}$, 
$\opnorm{\pi(\theta)}=
\opnorm{|\pi(\theta)|}=r(|\pi(\theta)|)$, where the latter is the spectral radius.

The terminology below is motivated by \cite{krantz}.  

\begin{proposition}\label{prop:Gomegaproperties}
If $\ome$ is a bounded weight on $\what{G}$ then $G_\ome$ is a compact subset
of $G_\Cee$ which contains $G$ and enjoys the following properties:

{\bf (i)} $G_\ome$ is {\rm $G$-Reinhardt}: for $s\iin G$ and
$\theta\in G_\ome$ we have 
$s\theta,\theta s\in G_\ome$; and

{\bf (ii)} $G_\ome$ is {\rm log-convex}: for $\theta,\theta'\in G_\ome^+=G_\ome\cap G_\Cee^+$
and $0\leq s\leq 1$, we have $\theta^s\theta'^{(1-s)}\in G_\ome$.

In particular, $(\theta,s)\mapsto \theta s:G_\ome^+\times G\to G_\om$ is a 
homeomorphism.

{\bf (iii)} If $\ome$ is symmetric then $G_\ome$ is inverse-closed.
\end{proposition}

\begin{proof}
(i) This is immediate from Proposition \ref{transl} and the fact that operators of 
translation are multiplicative.

(ii) It follows from  Proposition \ref{prop:GCproperties} that 
$\theta^s\theta'^{(1-s)}\in G_\Cee$ (though not necessarily in $G_\Cee^+$).
It is a standard fact of functional calculus that $\opnorm{\pi(\theta)^s}=
\opnorm{\pi(\theta)}^s$ for each $\pi\iin\what{G}$, and hence we have
\begin{align*}
\sup_{\pi\in\what{G}}\frac{\opnorm{\pi(\theta)^s\pi(\theta')^{(1-s)}}}{\ome(\pi)}
&\leq \sup_{\pi\in\what{G}}
\frac{\opnorm{\pi(\theta)}^s\opnorm{\pi(\theta')}^{1-s}}{\ome(\pi)} \\
&\leq\sup_{\pi\in\what{G}}
\frac{\max(\opnorm{\pi(\theta)},\opnorm{\pi(\theta')})}{\ome(\pi)}<\infty.
\end{align*}
It is immediate from (\ref{eq:spectrum}) that $\theta^s\theta'^{(1-s)}\in G_\ome$.

The homeomorphism identifying $G_\ome^+\times G\simeq G_\om$ is an immediate
consequence of (i), (ii) and Proposition \ref{prop:GCproperties} (iv).

(iii)
If $\ome$ is symmetric, then for $\theta\iin G_\Cee$, Corollary \ref{prop:ggproperties1}
shows that 
\[
\sup_{\pi\in\what{G}}\frac{ \opnorm{ \pi(\theta^{-1}) } }{\ome(\pi)}=
\sup_{\pi\in\what{G}}\frac{\opnorm{ \bar{\pi}(\theta) }}{\ome(\pi)},
\]
and hence
$\theta^{-1}\in G_\om$ if and only if $\theta\in G_\om$.
\end{proof}

If $\ome$ is bounded, then it is clear from Proposition \ref{prop:fomeisalg} that
$A_\ome(G)$ is semisimple.  If $\ome$ is not bounded this is not
as clear, but still true.  The following is motivated by \cite[2.8.2]{kanuith}.

\begin{theorem}\label{theo:semisimple}
The algebra $A_\ome(G)$ is semisimple.
\end{theorem}

\proof We first note that the constant function 1 is in $A_\ome(G)$. 
Hence the spectrum $G_\ome$ is non-empty.
Let  $\theta\in G_\ome$.

We note that if $u\in A_\ome(G)$ and $T\in A_\ome(G)^*$
then $T\cdot u\in A(G)$.  Indeed
\begin{align*}
\norm{T\cdot u}_A
&=\sum_{\pi\in\what{G}}\norm{\pi(T)\hat{u}(\pi)}_1d_\pi 
\leq\sum_{\pi\in\what{G}}\opnorm{\pi(T)}\norm{\hat{u}(\pi)}_1d_\pi \\
&\leq\sum_{\pi\in\what{G}}\frac{\opnorm{\pi(T)}}{\ome(\pi)}
\norm{\hat{u}(\pi)}_1d_\pi\ome(\pi) 
\leq \norm{T}_{A_\ome^*}\norm{u}_{A_\ome}<\infty.
\end{align*}
Thus for $\theta\iin G_\ome$, as above, $\theta\cdot u\in A(G)$ for $u\iin A_\ome(G)$.
Suppose $u\not=0$.
Since $u=\theta^{-1}\cdot(\theta\cdot u)$
(formally, in $\etrig$), we have that $\theta\cdot u\not=0$ in $A(G)$, and hence
there is some $s\iin G$ for which
\[
\langle u,s\theta\rangle=\theta\cdot u(s)\not=0.
\]
Since $s\theta\in G_\ome$ for all $s\iin G$ by Proposition \ref{prop:Gomegaproperties},
it follows that $A_\ome(G)$ admits no radical elements, and thus
is semisimple. \endpf


The following fact, which will be useful for the following
examples, is well-known.  If $\til{\om}:\En\to\Ree^{>0}$ is a {\it weight}, i.e.\
$\til{\ome}(n+m)\leq\til{\ome}(n)\til{\ome}(m)$, then 
\begin{equation}\label{eq:beurlingcond}
\rho_{\til{\ome}}=\lim_{n\to\infty}\til{\ome}(n)^{1/n}
=\inf_{n\in\En}\til{\ome}(n)^{1/n}.\end{equation}  
See, for example, \cite[A.1.26]{dales}.  Furthermore, if $\til{\om}$ is bounded
below, then $\rho_{\til{\ome}}\geq 1$.

\begin{example}\label{torus}\rm
Let $G=\Tee^n$ and $\om$ be a bounded weight on $\Zee^n\simeq\wh{G}$. Let 
$$\ell^1(\Zee^n,\om)=\left\{f:\Zee^n\to\Cee: \norm{f}_{1,\om}=\sum_{\mu\in\Zee^n}|f(\mu)|\om(\mu)<\infty \right\}$$
with convolution as multiplication.  The Fourier transform identifies
$\ell^1(\Zee^n,\om)$ with $A_\ome(\Tee^n)$ (formally, in the case that $\ome$ is not bounded).  
Let $\eps_1,\dots,\eps_n$ denote the
standard integer basis of $\Zee^n$.  Since $\del_{\eps_1},\dots,\del_{\eps_n}$ and their
inverses generate $\Zee^n$, any character $\chi\in\what{\ell^1(\Zee^n,\om)}\simeq \Tee^n_\ome$
is determined by the values $z_j=\chi(\del_{\eps_j})$, $j=1,\dots,n$.  Hence the Gelfand
transform converts $f\iin \ell^1(\Zee^n,\om)$ into a Laurent series
\[
\sum_{\mu\in\Zee^n}f(\mu)z^\mu\quad(z^\mu=z_1^{\mu_1}\dots z_n^{\mu_n}).
\]
These series converge, simultaneously for all $f\iin \ell^1(\Zee^n,\ome)$, if and only if
$|z^\mu|^k=|z^{k\mu}|\leq \ome(k\mu)$ for each $\mu\iin\Zee^n$.  Thus it follows from
an application of (\ref{eq:beurlingcond}) that
\[
\Tee^n_\ome\simeq\{z\in\Cee^n:
1/\rho_\ome(-\mu)\leq |z^\mu|\leq\rho_\ome(\mu)\text{ for all }\mu\in\Zee^n\}
\]
where $\rho_\ome(\mu)=\lim_{k\to\infty}\omega(k\mu)^{1/k}$.  If $n\geq 2$ the
family of defining inequalities can be simplified to choices of $\mu$ for which
$\mathrm{gcd}(\mu_1,\dots,\mu_n)=1$.  However if $n=1$ we obtain the usual
annulus of convergence with inner radius $1/\rho_\ome(-1)$ and outer radius
$\rho_\ome(1)$.

We observe that in the case that $n\geq 2$ and $\ome(\mu)=\lam^{\mu_1}$
for some $\lam>1$, then $\Tee^n_\ome\supsetneq\Tee^n$, but is not
an open subset of $\Cee^n$.  For any $\lam\iin(\Ree^{\geq 1})^n$
with $\lam_1\dots\lam_n>1$, the weight $\ome(\mu)=\lam^\mu$ defines
an exponential weight.

If $\alp>0$, the weight $\ome_\alp(\mu)=(1+\norm{\mu}_1)^\alp$ is a classical polynomial 
weight.  these weights will be generalised in (\ref{deftaS}).
\end{example}

\begin{example} \label{gom}\rm
Let $G=\T\rtimes \Z_2=\{(s,a): s\in \T, a\in \Z_2\}$ with multiplication 
$((s,a)(t,b)=(st^{a},ab)$, where $\Zee_2=\{1,-1\}$.  It is a straightforward
application of the ``Mackey machine'' that
$\wh{G}=\{1,\sigma,\pi_n, n\geq 1\}$, where
$\sigma$ is a one-dimensional representation given by $\sigma((s,a))=a$ and 
$\pi_n$, $n\geq 1$ is a two-dimensional representation defined by
$$
\pi_n((s,a))=\left(\begin{array}{cc}s^n&0\\0&s^{-n}\end{array}\right)
\left(\begin{array}{cc}0&1\\1&0\end{array}\right)^{(1-a)/2}.
$$
We have that the corresponding characters are 
$$\chi_\sigma((s,a))=a,\quad
\chi_{\pi_n}((s,a))=\left\{\begin{array}{cc}s^n+s^{-n},& a=1,\\
0,&a=-1.\end{array}\right.$$ 
Taking into account that
$\chi_{\pi\otimes\rho}=\chi_\pi\chi_\rho$, we obtain that
$1\otimes\sigma=\sigma$, $1\otimes\pi_n=\pi_n$,
$\sigma\otimes\sigma=1$, $\sigma\otimes\pi_n\approx\pi_n$,
$\pi_n\otimes\pi_m\approx \pi_{m+n}\oplus\pi_{|m-n|}$ if $m\ne n$ and
$\pi_n\otimes\pi_n\approx \pi_{2n}\oplus 1\oplus 1$.  Hence any weight $\om$ on $\wh{G}$
is governed only by the relations that $\ome(1),\ome(\sig)\geq 1$, and
$\til{\ome}(n)=\ome(\pi_n)$ defines
a weight on $\{0\}\cup\En$ which satisfies $\til{\ome}(0)=\ome(1)$ and
$\til{\ome}(|n-m|)\leq \til{\ome}(n)\til{\ome}(m)$.
Since $\bar{\pi}\approx\pi$ for each $\pi\iin\what{G}$, any weight
$\ome$ is automatically symmetric.

We observe that $G_\Cee\simeq\Cee^{*}\rtimes \Z_2$.  Indeed, we appeal to
\cite[Prop.\ 9]{cartwrightm} to see that 
$(G_\Cee)_e\simeq(G_e)_\Cee\simeq\Cee^*$, as $G_e\simeq\Tee$, and that
$G_\Cee/(G_\Cee)_e\simeq(G/G_e)_\Cee$, where $(G/G_e)_\Cee\simeq(\Zee_2)_\Cee\simeq
\Zee_2$ by definition of the complexification.  In particular we can identify
$G_\Cee^+=\{(\lam,1):\lam>0\}$ and we have $G_\Cee=GG_\Cee^+$.
Thus for $\lam>0$ and $s\iin\T$, the corresponding element
$\theta\iin G_\Cee$ is by $1(\theta)=1$, $\sig(\theta)=\sigma(s)$ and
$\pi_n(\theta)=\pi_n(s)\Lambda(n)$, where
$\Lambda(n)=\begin{pmatrix} \lambda^n&0\\0&\lambda^{-n}\end{pmatrix}$,  so 
$\Lambda(n)=\pi_n(|\theta|)$.

Let $\om :\wh{G}\to \R^{>0}$ be any weight and $\rho_\ome=\lim_{n\to\infty}
\ome(\pi_n)^{1/n}$.  For example if $\alpha>1$, let
$\om(1)=\om(\sigma)=1$ and $\om(\pi_n)=\alp^n$, and then $\rho_\ome=\alp$.
For $\theta$ as above, 
$\opnorm{\pi_n(\theta)}=\opnorm{\Lambda(n)}=\max(\lam^n,\lam^{-n})$.  
For each $n\iin \En$ (\ref{eq:beurlingcond}) implies $\tilde\ome(n)\geq\rho_\om^n$ and hence
\[
\sup_{n\in\En}\frac{\opnorm{\pi_n(\theta)}}{\tilde\ome(n)}
\leq\sup_{n\in\En}\frac{\max(\lam^n,\lam^{-n})}{\rho_\om^n}<\infty\quad
\Leftrightarrow\quad \max(\lam,\lam^{-1})\leq \rho_\om.
\]
Hence $G_\om\simeq\{(z,a):a\in\Zee_2, z\in\Cee, 1/\rho_\om\leq |z|\leq\rho_\om\}$.
In particular, if $\rho_\om>1$, then $G_\ome\supsetneq G$.
\end{example}

\begin{example}\label{sutwo}\rm
Let $G=\mathrm{SU}(2)$.  
We have that $\what{G}=\{\pi_n:n\in\{0\}\cup\En\}$ where $\pi_0=1$ and
$\pi_1$ is the standard representation.  The representations satisfy the
well-known tensor relations 
\[
\pi_n\otimes\pi_{n'}\approx
\bigoplus_{j=0}^{(n+n'-|n-n'|)/2}\pi_{|n-n'|+2j}.  
\]
Thus any weight
$\ome:\what{G}\to\Ree^{>0}$ is given by a weight $\til{\ome}:\{0\}\cup\En\to\Ree^{>0}$
which satisfies $\til{\ome}(|n-n'|+2j)\leq\til{\ome}(n)\til{\ome}(n')$ for
$j=0,\dots,(n+n'-|n-n'|)/2$. For example, any exponential weight
$\til{\ome}(n)=\lam^n$ for some $\lam>0$, suffices
 We have for every $n$ that $\bar{\pi}_n\approx\pi_n$,
so every weight is automatically symmetric.

It is well-known that $\mathfrak{su}(2)_\Cee=
\mathfrak{sl}_2(\Cee)$. It follows Proposition \ref{prop:GCproperties} (iii)
that $\mathrm{SL}_2(\Cee)\simeq\pi_1(i\mathfrak{su}(2))\pi_1(G)=\pi_1(G_\Cee)$.
Then, since $\pi_1$ generates $\what{G}$, we find
from (\ref{eq:tensdec}) that $\pi_1(G_\Cee)$ determines $G_\Cee$.  Hence
$\mathrm{SL}_2(\Cee)\simeq G_\Cee$.

Now given a weight $\ome$, let $\rho_\ome=\lim_{n\to\infty}\til{\ome}(n)^{1/n}$.
Any element of $\mathrm{SL}_2(\Cee)^+$ is, up to unitary equivalence, $\Lambda
=\begin{pmatrix}\lam & 0 \\ 0 & \lam^{-1}\end{pmatrix}$ for some $\lam>0$.
Taking successive tensor products $\Lambda^{\otimes n}
\approx\bigoplus_{j=0}^{(n+n'-|n-n'|)/2}\pi_{|n-n'|+2j}(\Lambda)$, we see by induction that
$\opnorm{\pi_n(\Lambda)}=\max(\lam^n,\lam^{-n})$.  Thus, using reasoning as
in the example above, and then (\ref{eq:spectrum}) and comments thereafter, we see that
\[
G_\ome\simeq\bigl\{x\in\mathrm{SL}_2(\Cee):\sig(|x|)=\{\lam,\lam^{-1}\},\,1/\rho_\ome\leq
\lam\leq\rho_\ome\bigr\}.
\]
\end{example}

\begin{example}\label{exponentialweight}\rm
Suppose $G_e$ is non-trivial so $G_\Cee\supsetneq G$.  If $\theta\in G_\Cee\setdif G$
let $\ome_\theta(\pi)=\opnorm{\pi(\theta)}=\opnorm{\pi(|\theta|)}$.  It is immediate from
(\ref{eq:tensdec}) that $\ome_\theta$ is a weight, and 
from Proposition \ref{prop:GCproperties} (i) we may choose $\theta$
to be positive.  It follows Corollary \ref{prop:ggproperties1} that $\ome_\theta$ is
symmetric if and only if $\opnorm{\pi(\theta)^{-1}}=\opnorm{\pi(\theta)}$
for each $\pi\in\what{G}$.  For the cases in Examples \ref{torus}, \ref{gom}
and \ref{sutwo} above,
these weights generalise the exponential weights.   
\end{example}

We can take advantage of Proposition \ref{prop:Gomegaproperties} to see that
$G_\ome$ contains some analytic structure when it is bigger than $G$.
We note that by Proposition \ref{prop:Gomegaproperties} $G_\om^+$ is
{\it logarithmically star-like} about $e$:  for $\theta\iin G_\om^+$ and
$0\leq s\leq 1$, $\theta^s=\theta^se^{1-s}\in G_\om^+$.  We will call
$\theta\iin G_\om^+\setdif\{e\}$ a {\it relative interior point} of $G_\om^+$ if 
$\theta^{1+\eps}\in G_\om^+$ for some $\eps>0$.  If $G_\om\supsetneq G$,
then $G_\om^+$ always admits relative interior points.

\begin{theorem}\label{theo:analyticstructure}
Suppose that $\ome$ is bounded and $G_\ome\supsetneq G$.  
Then for any relative interior point $\theta\iin G_\om^+\setdif\{e\}$ there are real 
numbers $\alp<\beta$ such that for every $u\iin A_\om(G)$,
$u_\theta(z)=\dpair{\theta^z}{u}$ defines 
a holomorphic function on $S_{\alp,\beta}=\{z\in\Cee:\alp <\re z<\beta\}$.
\end{theorem}

\begin{proof}
We let $\alp=\inf\{s\in\Ree:\theta^s\in G_\ome\}$ and
$\beta=\sup\{s\in\Ree:\theta^s\in G_\om\}$.
Proposition \ref{prop:GCproperties} (iv) provides $X\in i\gg$, 
for which $\exp(X)=\theta$.  We note for $z=s+it\iin S_{\alp,\beta}$,
that $\theta^z=\theta^s\exp(itX)\in G_\ome$.
Now, if $u\in\tripig$ then $u_\theta(z)=\Tr(\hat{u}(\pi)\exp(z\pi(X)))d_\pi$
defines a holomorphic function on $S_{\alp,\beta}$ for which
\[
\sup_{z\in S_{\alp,\beta}}|u_\theta(z)|\leq \sup_{\theta'\in G_\om}|\dpair{\theta'}{u}|
\leq\fomenorm{u}=\norm{\hat{u}}_1d_\pi\ome(\pi).
\]
Hence if we consider $u\iin A_\om(G)$, we see that
\[
u_\theta(z)=\sum_{\pi\in\what{G}}\Tr(\hat{u}(\pi)\pi(\theta)^z)d_\pi\ome(\pi)
\]
converges uniformly on $S_{\alp,\beta}$ and hence defines a holomorphic
function. 
\end{proof}

\begin{remark}\rm We will say that the unit $e$ of $G$
is a relative interior point of $G_\om^+$ if there is $\theta\iin G_\om^+\setdif\{e\}$
such that $\theta^{-\eps}\in G_\om^+$ for some $\eps>0$.
We note that if $\ome$ is symmetric and $G_\ome\supsetneq G$, then
it follows from  Corollary~\ref{prop:ggproperties1} if $\theta\in G_\om^+$ then $\theta^{-1}\in G_\ome^+$ showing that $e$ is a relative interior point of $G_\om^+$. 
Even in the case that $e$ is a relative interior point, the above procedure, applied to
$e$ produces only constant holomorphic functions.  If 
$\theta\iin G_\om^+\setdif\{e\}$ is a relative interior point, for which $\theta^{-\eps}\in G_\ome$,
the holomophic function $u_\theta$ satisfies $u_\theta(0)=u(e)$.
\end{remark}

\begin{definition}\rm
An involutive Banach algebra is called {\it symmetric} if for every
self-adjoint element $u$, the spectrum $\sigma(u)\subset\Ree$.
\end{definition}

If $\ome$ is a symmetric weight then $u\mapsto \bar{u}$ defines
an isometric involution on $A_\ome(G)$, where $\bar{u}(s)=\wbar{u(s)}$
for $s\iin G$.  Indeed, it is easy to check 
that for every $\pi\iin\what{G}$, and $u\iin\falomeg$ that
$\norm{\what{\bar{u}}(\pi)}_1=\norm{\hat{u}(\bar{\pi})}_1$.
It is then immediate from the definition of the norm that $\fomenorm{\bar{u}}=
\fomenorm{u}$.

\begin{theorem}\label{theo:symmetry}
Let $\ome$ be a symmetric weight on $G$.
The Beurling-Fourier algebra $\falomeg$ is symmetric if and only if
$G_\ome=G$.
\end{theorem}

\begin{proof}
If $G=G_\ome$, then it is obvious that $\falomeg$ is symmetric.

If $G_\om\supsetneq G$, then by Theorem \ref{theo:analyticstructure},
for any relative interior point $\theta\iin G^+_\om\setdif\{1\}$,
the function $u\mapsto u_\theta$ is a homomorphism from
$\falomeg$ into $\mathrm{Hol}(S_{\alp,\beta})$, the space
of holomorphic functions on an open strip $S_{\alp,\beta}$.
Since for $z\not=1$ but sufficiently close to $1$, $\theta^z\not=\theta$, there
is $u\iin\falomeg$ for which $\dpair{u}{\theta^z}\not=\dpair{u}{\theta}$.
Moreover, since $\falomeg$ is generated by its self-adjoint elements,
i.e.\ $2u=(u+\bar{u})+(u-\bar{u})$ for each $u$, there must be
a self-adjoint element $u$ for which $\dpair{u}{\theta^z}\not=\dpair{u}{\theta}$
for some $z$.  Hence $u_\theta$ is a non-constant holomorphic
function, whence $u_\theta(S_{\alp,\beta})$ is open in $\Cee$.
Since $u_\theta(S_{\alp,\beta})\subset\sigma(u)$, the latter cannot
be contained in $\Ree$.
\end{proof}

In the end of this section  we give general conditions on weights $\om$
for which the spectrum $G_\om$ of $A_\om(G)$ coincides with (is
different from) $G$.

\medskip

\noindent
{\bf Some functorial properties.}
The Beurling-Fourier algebras admit natural Beurling-Fourier algebras
when restricted to subgroups.  
If $H$ is a closed subgroup of $G$ and $\ome:\what{G}
\to\Ree^{>0}$ then we define $\ome_H:\what{H}\to\Ree^{>0}$ by
\[
\ome_H(\sig)=\inf_{\substack{\pi\in\what{G} \\ \sig\subset\pi|_H}}\ome(\pi).
\]

\begin{proposition}
$\ome_H$ is a weight on $\what{H}$.
\end{proposition}

\begin{proof} Let $\sig,\sig',\tau\in\what{H}$ with $\tau\subset\sig\otimes\sig'$ 
and $\eps>0$ be given.  
Find $\pi,\pi'\iin\what{G}$ such that
$\sig\subset\pi|_H$ with $\ome(\pi)<\ome_H(\sig)+\eps$ and
$\sig\subset\pi'|_H$ with $\ome(\pi')<\ome_H(\sig')+\eps$.  Then
$\tau\subset\pi\otimes\pi'|_H$ and hence
\[
\ome_H(\tau)\leq\inf_{\substack{\rho\subset\pi\otimes\pi' \\
\tau\subset\rho|_H}}\ome(\rho)
\leq\ome(\pi)\ome(\pi')\leq(\ome_H(\sig)+\eps)(\ome_H(\sig')+\eps).
\]
Since $\eps>0$ can be chosen arbitrarily and independently of
$\sig,\sig'$, it follows that
$\ome_H(\tau)\leq\ome_H(\sig)\ome_H(\sig')$.  \end{proof}

Let $\iota:H\to G$ denote the injection map which, as in Section \ref{sec:abslietheory},
extends to a homomorphism $\iota:\trih^\dagger\to\trig^\dagger$.

\begin{proposition}\label{prop:restrict}
{\bf (i)} The ``restriction'' map $u\mapsto u\comp\iota$ is a Banach
algebra quotient map from $\falomeg$ onto $\falomehh$.

{\bf (ii)} The spectrum $H_{\ome_H}$ of $\falomehh$ is isomorphic
to $G_\ome\cap \iota(H_\Cee)$.
\end{proposition}

\begin{proof}  {\bf (i)}
The extension $\iota:\trih^\dagger\to\trig^\dagger$, satisfies
\[
\pi\comp\iota(T)\simeq\bigoplus_{\substack{\sig\in\what{H} \\ \sig\subset\pi|_H}}
\sig(T)^{\oplus m(\sig,\pi)}
\]
where $m(\sig,\pi)$ is the multiplicity of $\sig$ in $\pi|_H$.
We have that $\iota\bigl(\falomehh^*\bigr)\subset\falomeg^*$.
Indeed $\sig\subset\pi|_H$ implies $\ome_H(\sig)\leq\ome(\pi)$ and hence
$\frac{\opnorm{\sig(T)}}{\ome_H(\sig)}\geq\frac{\opnorm{\sig(T)}}{\ome(\pi)}$
so we have
\[
\dfomenorm{\iota(T)}=\sup_{\pi\in\what{G}}
\frac{\opnorm{\pi\comp\iota(T)}}{\ome(\pi)}=\sup_{\pi\in\what{G}}
\max_{\substack{\sig\in\what{H} \\ \sig\subset\pi|_H}}\frac{\opnorm{\sig(T)}}{\ome(\pi)}
\leq\sup_{\sig\in\what{H}}\frac{\opnorm{\sig(T)}}{\ome_H(\sig)}.
\]
Hence $\iota|_{\falomehh^*}:\falomehh^*\to\falomeg^*$ is
contractive. Moreover, $\iota|_{\falomehh^*}$ is an isometry.  Indeed, given
$T\iin\falomehh^*$ and $\eps>0$, there is $\sig\iin\what{H}$ such
that $\frac{\opnorm{\sig(T)}}{\ome_H(\sig)}>\dfomehnorm{T}-\eps$.
Moreover there is $\pi\in\what{G}$ such that $\sig\subset\pi|_H$ and
$\ome(\pi)<\ome_H(\sig)+\eps$.  Then
\[
\fomenorm{\iota(T)}\geq\frac{\opnorm{\sig(T)}}{\ome(\pi)}
>\frac{\opnorm{\sig(T)}}{\ome_H(\sig)+\eps}
\geq\bigl(\dfomehnorm{T}-\eps\bigr)\frac{1}{1+\frac{\eps}{\ome_H(\sig)}}
\]
from which it follows that $\iota|_{\falomehh^*}$ is an isometry.

Since $\trig$ is dense in $\falomeg$, it follows that the preadjoint
$u\mapsto u\comp \iota=u|_H$ of $\iota|_{\falomehh^*}$ extends to a quotient
map from $\falomeg$ onto $\falomehh$.

{\bf (ii)} It is noted in \cite[Cor.\ 3]{cartwrightm} that the 
map $\iota|_{H_\Cee}:H_\Cee\to G_\Cee$ is a topological embedding.
Also, $H_{\ome_H}=H_\Cee\cap\falomehh^*$.
Hence $\iota(H_{\ome_H})= G_\Cee\cap\falomeg^*=G_\ome$.
\end{proof}

The connected component of the identity warrants particular consideration.

\begin{corollary}\label{cor:connectedcomp}
The connected component of $G_\om$ containing
$e$, $(G_\ome)_e$ is naturally isomorphic with $(G_e)_{\ome_{G_e}}$.  
In particular, in the case that $\ome$ is bounded,
$G_\ome\supsetneq G$ if and only if $(G_\ome)_e\supsetneq G_e$.
\end{corollary}

\proof  It follows from Proposition \ref{prop:Gomegaproperties} (ii) that
$G_\om^+\subset (G_\ome)_e$.  The same proposition shows that
$(\theta,s)\mapsto\theta s:G_\ome^+\times G\to G_\ome$ is a homeomorphism,
and hence $G_\ome^+ G_e= (G_\ome)_e$.  However, since
$G_\Cee^+ G_e=\iota(G_e)_\Cee$, we get, from Proposition \ref{prop:restrict} above, that
$G_\ome^+G_e=G_\om\cap\iota(G_e)_\Cee=(G_e)_{\ome_{G_e}}$.

If $\ome$ is bounded, 
we have $G_\ome\supsetneq G$ if and only if $G_\om^+\supsetneq\{e\}$.
Hence this condition is equivalent to $(G_\ome)_e\supsetneq G_e$.  \endpf

Let $N$ be a normal subgroup of $G$ and $q:G\to G/N$ be the quotient map.
The map $\pi\mapsto\pi\comp q:\what{G/N}\to\what{G}$ clearly preserves
decomposition into irreducible components.
Thus if $\ome:\wh{G}\to\Ree^{>0}$ is a weight
we may define a weight $\ome^N:\what{G/N}\to\Ree^{>0}$ by
\[
\ome^N(\pi)=\ome(\pi\comp q).
\]
As above, we let $\iota:N\to G$ denote the injection which extends naturally
to a map $\iota:\tri{N}^\dagger\to\trig^\dagger$.  We note that since $N$ is normal in $G$,
$\n=\{X\in\gg:\exp(tX)\in N\text{ for all }t\in\Ree\}$ is a Lie ideal in
$\gg$, whence $\n_\Cee$ is a Lie ideal in $\gg_\Cee$, from which it follows
that $N_\Cee\cong\iota(N_\Cee)=N\exp(i\n)$ is normal in $G_\Cee$.

\begin{proposition}\label{prop:quotient}
{\bf (i)} The map $u\mapsto u\comp q:A_{\om^N}(G/N)\to\falomeg$ is an isometric
homomorphism.

{\bf (ii)}  On $G_\om$, let $\theta\sim_{N_\Cee}\theta'$ if $\theta^{-1}\theta'\in\iota(N_\Cee)$ 
Then the quotient space $G_\om/N_\Cee$ 
may be identified with a closed subset of $(G/N)_{\ome^N}$.
\end{proposition}

\begin{proof}{\bf (i)} 
Define $P_Nu(s)=\int_N u(sn)dn$.  Then, since  by Proposition~\ref{transl} translations
are isometries on $\falomeg$, we have that $P_N$ defines a bounded linear operator
on $\falomeg$.  Moreover, $P_N^2=P_N$ and $P_N\bigl(\falomeg\bigr)=
\falome{G\!:\!N}$, the subalgebra of elements constant n cosets of $N$.  
It remains to prove the latter is isometrically isomorphic
to $\falomegn$.

Let us note that if $\pi\in\what{G}\setdif(\what{G/N}\comp q)$
then $\pi|_N$ never contains the trivial representation of $N$.  Indeed if for
$\xi\in\fH_\pi$ we have $\pi(n)\xi=\xi$ for all $n\iin N$, then for any $s\iin G$
and any $n\iin N$, we have $\pi(n)\pi(s)\xi=\pi(s)\pi(s^{-1}ns)\xi=\pi(s)\xi$.
Then either $\xi=0$, or $\xi$ is a cyclic vector for $\pi(G)$, in which case
$\pi(n)=I$ for all $n\iin N$, but this 
contradicts our assumption about $\pi$.

Thus if $u\in\falomeg$ we have for $\pi\iin\what{G}$, $s\iin G$
\begin{align*}
P_N\Tr(\hat{u}(\pi)\pi(s))&=\Tr\left(\hat{u}(\pi)\pi(s)\int_N\pi(n)dn\right) \\
&=\begin{cases} 0 &\iif \pi\not\in\what{G/N}\comp q \\
\Tr(\hat{u}(\pi)\pi(s)) &\iif \pi\in\what{G/N}\comp q \end{cases}
\end{align*}
by the Schur orthogonality relations.  Thus
\[
\falome{G\!:\!N}=\left\{u\in\etrig:
\begin{matrix} \hat{u}(\pi)=0\ffor\pi\in\what{G}\setdif\what{G/N}\comp q
\aand \\ \sum_{\pi\in\what{G/N}\comp q}\norm{\hat{u}(\pi)}_1d_\pi\ome(\pi)<\infty
\end{matrix}\right\}
\]
which is clearly isometrically isomorphic to $\falomegn$.

{\bf (ii)}  We consider the extended map $q:\trig^\dagger\to\tri{G/N}^\dagger$.
We have that $\ker q|_{G_\Cee}=\iota(N_\Cee)$.  Indeed, we note that
since $q(N)=\{e\}$, $q(\n)=\{0\}$, and hence $q\comp\iota(N_\Cee)=q(N\exp(i\n))
=\{e\}$.  Conversely, if $\theta=s|\theta|\in \ker q|_{G_\Cee}$, then
$q(s)q(|\theta|)$ is the polar decomposition of $q(\theta)=e$, hence
$q(s)=e=q(|\theta|)$.  Thus $s\in N$.  Moreover we write $|\theta|=\exp(iX)$
for some $X\iin\gg$.  We see for $t,t_0\in\Ree$ that $e=q(|\theta|)^{tt_o}=q(\exp(itt_0X))
=\exp(itq(t_0X))$, and, taking derivative at $t=0$ we obtain that $iq(t_0X)=0$.
Thus we see that $X\in\n$, hence $\theta=s\exp(iX)\in \iota(N_\Cee)$.

Now, from (i) above, $q(G_\ome)$ will be a closed subset of $(G/N)_{\ome^N}$.
We see that for $\theta,\theta'\iin G_\ome\subset G_\Cee$, $q(\theta)=q(\theta')$
if and only if $q(\theta^{-1}\theta')\in\iota(N_\Cee)$, i.e.\ $\theta\sim_{N_\Cee}
\theta'$.  
%
\end{proof}

That $q:G\to G/N$ extends to an open quotient map $q|_{G_\Cee}:G_\Cee\to
(G/N)_\Cee$ is noted in \cite[Cor.\ 3]{cartwrightm} and requires the somewhat
delicate lifting one-parameter subgroup result \cite[Theo.\ 4]{mckennon}.  It is 
unclear that this result preserves the rate of growth of positive elements.

\begin{conjecture}\label{conj:quot}
In (ii), above, we have that $G_\ome/\sim_{N_\Cee}=(G/N)_{\ome^N}$.  In particular,
if for some Lie quotient $G/N$ of $G$, $(G/N)_{\ome^N}\supsetneq G/N$, then
$G_\ome\supsetneq G$.
\end{conjecture}






\medskip
\noindent
{\bf Growth of weights.} 
We wish to find conditions on the weight $\ome$ which characterise
when $G_\ome\supsetneq G$ and $G_\ome =G$.

  We begin with some notation.  If $S\subset\what{G}$
we let
\[
S^{\otimes n}=\{\pi\in\what{G}:\pi\subset\sig_1\otimes\dots\otimes\sig_n\wwhere
\sig_1,\dots,\sig_n\in S\},\quad\langle S\rangle=\bigcup_{n\in\En}S^{\otimes n}.
\]
We say that $\what{G}$ is finitely generated if $\what{G}=\langle
S\rangle$ for some finite $S\subset\what{G}$.

\begin{proposition}
 $G$ is a Lie group if and only
if $\what{G}$ is finitely generated.
\end{proposition}
 \begin{proof}
This is \cite[(30.48)]{hewittrII}.
 \end{proof}

We  let  for any continuous unitary representation $\rho$ of $G$, and any
finite subset $S$ of $\what{G}$ 
\[
\ome(\rho)=\sup_{\sig\in\what{G},\sig\subset\rho}\ome(\sig)\aand
\ome(S)=\sup_{\sig\in S}\ome(\sig)
\]
so that $\ome(\rho)=\ome(\{\pi\in\what{G}:\pi\subset\rho\})$.  We note that
if $S$ is a finite subset of $\what{G}$, then $S^{\otimes (n+m)}
=S^{\otimes n}\otimes S^{\otimes m}$, i.e.\ any $\pi\iin S^{\otimes (n+m)}$
may be realised as a subrepresentation of $\pi'\otimes\pi''$ for some $\pi'\iin 
S^{\otimes n}$ and $\pi''\iin S^{\otimes m}$. 
Hence the function $\til{\ome}:\En\to\Ree^{>0}$
given by $\til{\ome}(n)=\ome(S^{\otimes n})$ is a weight.  Thus we can
appeal to (\ref{eq:beurlingcond}) to define
\[
\rho_\om(S)=\lim_{n\to\infty}\ome(S^{\otimes n})^{1/n}
\]
and for a single $\pi\iin\what{G}$, we define $\rho_\om(\pi)
=\lim_{n\to\infty}\ome(\pi^{\otimes n})^{1/n}$.

We say that $ \om $ is {\it nonexponential} if for every $ \pi\in\wh G $
\begin{eqnarray}
 \rho_\om(\pi)= 1
\end{eqnarray}
and we say that $\ome$ has {\it exponential growth} otherwise.

\begin{proposition}\label{prop:exgrowchar}
Let $\ome$ be a bounded weight on $\what{G}$.  Then the following are equivalent:

{\bf (i)} $\ome$ has exponential growth; and

{\bf (ii)} there is some finite subset $S$ of $\what{G}$ for which
$\rho_\ome(S)>1$.

\noindent Further, if $G$ is a Lie group then (i) and (ii) are equivalent to:

{\bf (iii)} $\rho_\ome(S)>1$ for every generating set $S\subset\what{G}$.
\end{proposition}


\begin{proof}
Recall that a bounded weight always has $\rho_\ome(\pi)\geq 1$;
see remark after (\ref{eq:beurlingcond}).
That (i) implies (ii) is obvious. 
In the case that $G$ is a Lie group it is obvious that (iii) implies (ii).
For a Lie group $G$, if $S'$ is a generating set, then for some $m$,
${S'}^{\otimes m}\supset S$ and hence 
\[
\rho_\ome(S')=\lim_{n\to\infty}\ome({S'}^{\otimes n})^{1/n}\geq
\lim_{n\to\infty}\ome({S'}^{\otimes mn})^{1/mn}=\rho_\ome({S'}^{\otimes m})^{1/m}
\geq \rho_\ome(S)^{1/m}.  
\]
Hence (ii) implies (iii).
It remains to show that (ii) implies (i), in general.

Let $S=\{\sig_1,\dots,\sig_m\}$.  Suppose $\rho_\ome(S)>1$.  By
(\ref{eq:beurlingcond}) we have that
$\ome(S^{\otimes n})\geq\rho_\ome(S)^n$,
and hence there is a sequence $(\pi_n)_{n\in\En}$ such that
\[
\pi_n\in S^{\otimes n}\text{ for each }n,\aand \ome(\pi_n)\geq\rho_\ome(S)^n.
\]
Then for each $n$ there are $l_{1,n},\dots, l_{m,n}
\iin\{0\}\cup\En$ such that $\pi_n\in\sig_1^{\otimes l_{1,n}}\otimes\dots\otimes
\sig_m^{\otimes l_{m,n}}$ and $l_{1,n}+\dots +l_{m,n}=n$.  We have
\[
\ome(\pi_n)\leq\ome(\sig_1^{\otimes l_{1,n}})\dots\ome(\sig_m^{\otimes l_{m,n}})
\]
It follows from the ``pigeon-hole principle''
that for some $j=1,\dots,m$ that there is a sequence
$n_1<n_2<\dots$ for which $\ome(\pi_{n_k})^{1/m}\leq
\ome(\sig_j^{\otimes l_{j,n_k}})$.  Since 
$\rho_\ome(S)^{n_k/m}\leq\ome(\sig_j^{\otimes l_{j,n_k}})$, we may assume
$n_1,n_2,\dots$ are chosen so $l_{j,n_1}<l_{j,n_2}<\dots$.
Thus, since $l_{j,n_k}\leq n_k$ and $\ome(\sig_j^{\otimes l_{j,n_k}})>1$, we have
\[
\ome(\pi_{n_k})^{1/mn_k}\leq\ome(\sig_j^{\otimes l_{j,n_k}})^{1/n_k}
\leq\ome(\sig_j^{\otimes l_{j,n_k}})^{1/ l_{j,n_k}}
\]
and we  find
\[
1<\rho_\ome(S)^{1/m}=\lim_{k\to\infty}\ome(\pi_{n_k})^{1/mn_k}\leq
\lim_{k\to\infty}\ome(\sig_j^{\otimes  l_{j,n_k}})^{1/ l_{j,n_k}}
\]
where the latter is $\rho_\ome(\sig_j)$, again by (\ref{eq:beurlingcond}).
\end{proof}

\begin{example}\rm

(1) Consider exponential weights $\ome_\theta$ of Example \ref{exponentialweight}.
As $\theta\in G_\Cee$ we may appeal to (\ref{eq:coproduct})
and (\ref{eq:grouplike}) to see that
\[
\ome_\theta(\pi^{\otimes n})=\max_{\sig\subset\pi^{\otimes n}}\opnorm{\sig(\theta)}
=\opnorm{\pi^{\otimes n}(\theta)}=\opnorm{\pi(\theta)^{\otimes n}}=\opnorm{\pi(\theta)}^n
\]
and hence $\rho_{\ome_\theta}(\pi)=\opnorm{\pi(\theta)}$.  In particular such a weight is
bounded only if $\inf_{\pi\in\what{G}}\opnorm{\pi(\theta)}\geq 1$, and
of nonexponential growth only if $\opnorm{\pi(\theta)}=1$ for all $\pi\iin\what{G}$.

Corollary \ref{prop:ggproperties1} shows that $\ome_\theta$ is symmetric
exactly when, for each $\pi$, the smallest and largest eigenvalues
$\mu_\pi,\lam_\pi$ of $|\pi(\theta)|$ satisfy $\lam_\pi=\mu_\pi^{-1}$.


(2) Let $G$ be a compact group and let $\om(\pi)=d_\pi$, $\pi\in\wh
G$, be the dimension weight. The weight $\om$ is nonexponential by
Example~\ref{dimension} and Proposition~\ref{poly-sub} below.
\end{example}

Other examples of weights of nonexponential growth are given in the
next section.

\begin{proposition}\label{polygivesnoth}
Let $ \om $ be a nonexponential symmetric weight on $ \wh G $. Then $ G_\om=G $.
 \end{proposition}
\begin{proof}
Assume $ G_\om\ne G $. Then by Proposition \ref{prop:Gomegaproperties}
there exists $ \theta\in G_\om^+ $ such that
$ \sup_{\pi\in\wh G}\frac{\noop{\pi(\theta)}}{\om(\pi)}\leq 1 $ and $ \noop{\pi(\theta)}>1 $
 for some $ \pi\in\wh G $. Then
\[
\om(\pi^{\otimes n})^{1/n}= \sup_{\si\subset \pi^{\otimes n}}\om(\si)^{1/n} 
\geq \sup_{\si\subset \pi^{\otimes n}}\noop{\sig(\theta)}^{1/n}
=\noop{\pi(\theta)^{\otimes n}}^{1/n}=\noop{\pi(\theta)}
\]
giving $ \lim_{n\to\infty} \om(\pi^{\otimes n})^{1/n}>1 $,   a contradiction.
 \end{proof}

\begin{question}
Is it true in general that if  a weight $\om$ is exponential then $G_\om\ne G$?
\end{question}


\newtheorem{prop}[subsection]{Proposition}

\newtheorem{remarks}[subsection]{Remarks}

\newtheorem{notations}[subsection]{Notations}
\def\C{\mathbb C}
\def\R{\mathbb R}
\def\S{{\mathcal S} }
\def\g{\mathfrak g}
\def\h{\mathfrak h}
\def\c{\mathfrak c}
\def\a{\mathfrak a}
\def\k{\mathfrak k}
\def\u{\mathfrak u}
\def\z{\mathfrak z}
\def\C{\mathbb C}
\def\R{\mathbb R}
\def\I{\mathbb I}
\def\N{\mathbb N}
\def\al{\alpha}
\def\be{\beta}
\def\DE{\Delta}
\def\de{\delta}
\def\rh{\rho}
\def\ga{\gamma}
\def\GA{\Gamma}
\def\va{\varepsilon}
\def\LA{\Lambda}
\def\la{\lambda}
\def\OM{\Omega}
\def\om{\omega}
\def\var{\varphi}
\def\sp#1#2{\langle{#1},{#2}\rangle}
\def\cc{\mathfrak{c}}
\def\g{\mathfrak{g}}
\def\a{\mathfrak{a}}
\def\b{\mathfrak{b}}
\def\h{\mathfrak{h}}
\def\k{\mathfrak{k}}
\def\q{\mathfrak{q}}
\def\p{\mathfrak{p}}
\def\n{\mathfrak{n}}
\def\m{\mathfrak{m}}
\def\l{\mathfrak{l}}
\def\j{\mathfrak{j}}
\def\s{\mathfrak{s}}
\def\t{\mathfrak{t}}
\def\z{\mathfrak{z}}
\def\u{\mathfrak{u}}
\def\r{\mathfrak{r}}

\def\noi{\noindent}

\def\ga{\gamma}
\def\la{\lambda}
\def\ve{\varepsilon}
\def\si{\sigma}
\def\om{\omega}
\def\et{\eta}
\def\va{\varphi}
\def\om{\omega}
\def\ep{\epsilon}
\def\vp{\varphi}
\def\PH{\Phi}
\def\ps{\psi}
\def\ph{\phi}
\def\ch{\chi}
\def\Ga{\Gamma}

\def\N{\mathbb{N}}
\def\Z{\mathbb{Z}}
\def\R{\mathbb{R}}
\def\C{\mathbb{C}}
\def\ad{\rm{\, ad\, }}
\def\Ad{\rm{\, Ad \,}}
\def\Om{\Omega}
\def\ol#1{\overline{#1}}


\def\R{{\mathbb R}}
\def\C{{\mathbb C}}
\def\N{{\mathbb N}}
\def\Q{{\mathbb Q}}
\def\Z{{\mathbb Z}}
\def\T{{\mathbb T}}
\def\I{{\mathbb I}}

  \def\Id{{\mathbb I}}
\def\A{{\mathcal A}}
\def\B{{\mathcal B}}
\def\D{{\mathcal D}}
\def\F{{\mathcal F}}
\def\E{{\mathcal E}}
\def\H{{\mathcal H}}
\def\K{{\mathcal K}}
\def\L{{\mathcal L}}
\def\M{{\mathcal M}}
\def\RR{{\mathcal R}}
\def\cS{{\mathcal S}}
\def\T{{\mathcal T}}
\def\X{{\mathcal X}}
\def\U{{\mathcal U}}
\def\V{{\mathcal V}}
\def\W{{\mathcal W}}
\def\Z{{\mathcal Z}}
\def\ad{{\rm ad}}
\def\tr{{\rm tr}}

\def\iy{\infty}

\def\noin{\noindent}

\def\ol#1{\overline{#1}}
\def\ul#1{\underline{#1}}

\def\cds#1#2{#1,\cdots,#2}
\def\cdsp#1#2{\{#1,\cdots,#2\}}
\def\hb#1{\hbox{#1}}
\def\val#1{\vert #1\vert}

\def\no#1{\Vert #1\Vert }
\def\opno#1{\Vert #1\Vert_{\mathrm op} }
\def\pa#1{\{#1\}}
\def\vatw#1{\vert #1\vert_{L^2}}

\def\ind#1#2{\hb{ind}_{#1}^{#2}}
\def\CC#1{ C_c(#1)}

\def\ker#1{\hb{ker}(#1)}

\def\res#1{_{\vert #1}}
\def\inv{^{-1}}

\def\es{\emptyset}

\def\vs #1#2#3{\vskip #1#2#3 cm{\n}}
\def\hs #1#2#3{\hskip #1#2#3 cm}

\def\me{\medskip\noindent}
\def\hb #1{\hbox{#1}}

\def\un#1{\underline{#1}}

\def\nn{\nonumber}

\def\cdp#1#2{(#1,\cdots,#2)}

\def\cda#1#2{\{#1,\cdots,#2\}}

\def\cds{\cdots}

\def\hb#1{\hbox{#1}}
\def\val#1{\vert #1\vert}

\def\pa#1{\{#1\}}

\def\ker#1{\hb{ker}(#1)}

\def\im#1{\hb{im}(#1)}

\def\sp#1#2{\langle #1,#2\rangle }

\def\ca#1{{\mathcal #1}}

\def\dim#1{\hb{dim}(#1)}
\def\Log#1{\rm{Log}(#1)}

\def\bo#1{{\bf #1}} 
\def\ele{\'el\'ement}
\def\HS{{\mathcal H\mathcal S}}
\def\L1#1{L^1(#1)}
\def\Im{\mathrm{\, Im \,}}
\def\Re{\mathrm{\, Re\, }}
\def\ind{\mathrm{\, ind\, }}

\def\stacksx{\tilde\xi}
\def\stacksf{\stackrel{\sim}{f}}
\def\xiphi{\xi\otimes^p_{\rho}\phi}
\def\etapsi{\eta\otimes^{p'}_{\rho}\psi}
\def\pl{\,\parallel\;}
\def\lef({\left(}
\def\rig){\right)}
\def\lan{\langle}
\def\ran{\rangle}

\def\span#1#2{\LA ({#1,#2})}

\section{Polynomial weights}\label{polynomial}

In this section we introduce the polynomial weights which are of fundamental 
importance.  For ease {\it we will always assume that a weight $\om$
on $\what{G}$ is symmetric}.  In particular these weight are bounded 
and thus $G\subset G_\ome$.


\medskip\noindent{\bf Definition and basic theory.}
The following description of the dual space of a connected compact
Lie group $G$ has been taken from \cite{wa}. Let $\g$ be the Lie
algebra of $G$. Then $\g=\z\oplus \g_1$ with $\z$ the center of $\g$
and $\g_1=[\g,\g]$ a compact Lie algebra. Let $\sp{\cdot}{\cdot}$ be
an inner product on $\g$ satisfying (1) $\sp {\g_1} \z=(0)$ and (2)
$\sp{\cdot}{\cdot}_{\vert {\g_1\times\g_1}}=-B_{\g_1}$ (here $B_\k$
denotes the Killing form of a Lie algebra $\k$).

Let $X_1,\cdots,X_n$ be an
orthonormal basis of  $\g$, such that $\{X_1,\cds, X_r\}$ is a basis
of $\z$. Set 
\begin{equation}\label{eq:casimir}
\OM=\sum_i X_i^2\in \trig^\dagger. 
\end{equation}
Then $\OM$ is independent
of the choice of the orthonormal basis of $\g$ and $\OM$ is central
in $\trig^\dagger$.  (Normally, the element $\OM$ is defined in the 
universal enveloping algebra $U(\g)$ of $\g$, but for our purposes
it is sufficient regard its image in the associative algebra $\trig^\dagger$.)

Let $\t$ be a maximal abelian subalgebra of $\g_1$ and let
$T=\exp{\t}$. Let also $\la_1,\cdots,\la_r$ be complex valued linear
forms on $\z$ defined by $\la_j(X_i)=2\pi(-1)^{1/2}\de_{i,j}$. Let
$P$ be a Weyl chamber of $T$. Let $\LA_1,\cdots,\LA_l$ be defined by
$ \frac{2\LA_i(H_{\al_j})}{\al_j(H_{\al_j})}=\de_{i,j}$, where
$\al_1,\cdots,\al_l$ are the simple roots relative to $P$ and the
$H_{\al_j}$ the corresponding vectors in $\t$. To every $\ga$ in the
dual space $\widehat G$ of $G$ corresponds a unique element
$\LA_\ga=\sum_{i}n_i \la_i+\sum_j m_j \LA_j$ with the $n_i$ integers
and the $m_j$ nonnegative integers. Set $$\Vert \ga\Vert=\sup
\{\val{\LA_\ga(X)}:\norm{X}=1, X \in P\}
$$ and  $$ \no\ga_1=\sum_i\val{n_i}+\sum_j m_j .$$

We let now $ \pi_{i}=\ch_i $ be the character  of the group $ G $
associated to the highest weight $ \la_i,  i=1\cdots ,r;   $ and let
$ \ga_j $  be the irreducible representation associated to the
weight $ \LA_j, j=1, \cdots ,l $.

Let $ S=\{\pm\ch_i,\ga_j, i=1,\cdots,r, j=1, \cdots ,l \} $.
It is well known that for two irreducible representations $ \pi_\LA,\pi_M $ of $
G $ the tensor product representation $ \pi_\LA\otimes\pi_M $ contains the
representation $ \pi_{\LA+M} $ exactly once and all its irreducible components
$ \pi_N $ satisfy the relation $ N\leq \LA+M $, i.e.,  $ 0\leq N(X)\leq
\LA(X)+M(X) $ for every $ X $ in the Weyl chamber (see \cite[p. 111]{knapp}).
Note that $ \norm{\ga}\leq\norm{\pi} $ if $ \LA_\ga\leq\LA_\pi $ and
since $ \norm{\cdot} $ and $ \norm {\cdot}_1$ are equivalent we have also  $
\norm{\ga}_1\leq C\norm{\pi}_1  $ for some constant $ C $.

Therefore $ S $ generates $ \wh{G} $.  This allows us to define the function $
\ta_S $ on $ \widehat{ G} $
\begin{eqnarray}\label{deftaS}
 \ta_S(\pi)= k,\text{ if }\pi\in S^{\otimes k}\setminus S^{\otimes (k-1)}.
\end{eqnarray}

 Now for  any  highest weight $ \LA_\ga=\sum_{i}n_i \la_i+\sum_j m_j \LA_j$
corresponding to the irreducible representation $ \ga $ of $ G $ we see that
\begin{eqnarray}
 \nn \ga\subset \prod_{i}\ch_i^{n_i}\otimes\prod_j \ga_j^{\otimes m_j}\subset S^{\otimes \no
{\ga}_1}
\end{eqnarray}
and
\begin{eqnarray}
 \nn \ga\not\subset S^{\otimes (\no {\ga}_1-1)}.
\end{eqnarray}
This shows that we have the relation:
\begin{eqnarray}\label{tas=no}
 \ta_S (\gamma)=\no{\gamma}_1, \ga\in\widehat{G}.
\end{eqnarray}
We shall work with the fundamental polynomial weight
\begin{eqnarray}\label{tasdef}
 \nn \om_S=1+\tau_S.
\end{eqnarray}
Then for every power $ \al\ (\al\in\R^{>0}) $ the function $ \om_\al=\om_S^
\al$ is
also a weight on $ \wh{G} $.
We observe that if $S$ and $S'$ are both generating sets for $\what{G}$, then
there are constants $k_1,k_2$ such that $k_1\ome_S\leq \ome_{S'}\leq k_2\ome_S$.
For example, if $k$ is such that $S'\subset S^{\otimes k}$, then $\tau_{S'}\leq k\tau_S$
and hence $\ome_{S'}\leq k\ome_S$.  Hence $A_{\ome_S}(G)=A_{\ome_{S'}}(G)$.

We know from \cite[Lemma 5.6.4]{wa}  (with the notations of that
lemma)  that for every $\ga\in\widehat G$
\begin{equation}\label{eq:casimireig}
-\ga(\OM)=(\sp{\LA_\ga+\rho}{\LA_\ga+\rho}-
\sp\rho\rho)\I_{\mathcal H_\ga}=: c(\ga)I_\ga
\end{equation}
where $\rho$ is half the sum of the positive roots of $G$ related
to the Weyl chamber of $T$.  Then by \cite[Lemma 5.6.6]{wa}, there
are positive constants $c_1,c_2$ such that
\begin{equation}\label{equi}
c_1\Vert \gamma\Vert^2\leq c_1\Vert \ga\Vert_1^2
  \leq c(\gamma)\leq c_2\Vert \gamma\Vert^2\leq c_2\Vert \ga\Vert_1^2
  \end{equation}
and by \cite[Lemma 5.6.7]{wa}, the series
\begin{equation}\label{series}
\sum_{\gamma \in\widehat G} d_\gamma^2 (1+\Vert
\gamma\Vert^2_1)^{-s}
\end{equation}
converges if $s> \frac {d(G )}2$. Here $d(G)$ denotes the
the dimension of the group $G$.

We say that function $ u: G\to\C $ is \textit{$ 2n $-times $ \OM
$-differentiable}  if $ u \in L^2_{\om_{2n}}, (n\in \R^{>0})$. The
space $L^2_{\om_{2n}}(G)$ coincides with the space $$\{g\in
L^2(G):(1-\Om)^ng\in L^2(G)\}.$$  In fact $ u $ is on the domain of
the operator $ (1-\OM)^n $, since for every $ \ga\in \hat
 G$
\begin{eqnarray}
\nonumber \ga(\OM(u))&=& -c(\ga)\ga(u)
\end{eqnarray}
 and so by (\ref{tas=no})
 \begin{eqnarray*}
 \nonumber \sum_{\ga\in\wh{G}}d_\ga (1+c(\ga))^n\Vert\ga(u)\Vert^2_2\leq
\sum_{\ga\in\wh{G}}c_2d_\ga w_S^{2n}(\ga)\Vert\ga(u)\Vert^2_2
\\\leq
\sum_{\ga\in\wh{G}}c_1d_\ga (1+c(\ga))^n\Vert\ga(u)\Vert^2_2.
\end{eqnarray*}
With some abuse of notation we  write here $\gamma(f)$ instead of
$\hat f(\gamma)$ for $f\in L^2(G)$ and $\gamma\in\wh G$.



\begin{proposition}\label{mdiff} Let $\alpha>0$. Then 
$L^2_{\om_{2n}}(G)\subset A_{\om_\alpha }(G)  $ if $
n>\frac{d(G)}{4}+\frac{\alpha}{2} $.
 \end{proposition}
\begin{proof}

Using the Plancherel theorem, one can find an $L^2$-function $E_m$
on $G$, such that $$
  \ga(E_m)=\frac{1}{(1+c(\ga))^{m}}I_\ga, \forall
\ga\in\widehat
  G,$$
  for all real $m>\frac{d(G)} 4$. 
Then for $g\in L^2_{\om_{2n}}(G)$, $ n>\frac{d(G)}{4}+\frac{\alpha}{2}$ we have
\begin{equation}\label{inverse1}E_n*(1-\Omega)^n g= g.
\end{equation}
Since $(1-\Omega)^n g\in L^2(G)$ to see that $ g\in
A_{\om_\alpha}(G) $ it is enough to prove that $ E_n\in
L^{2}_{\om_\alpha^2}(G) $. By (\ref{equi}) we have
\begin{eqnarray}
 \nn \sum_{\ga}
d_\ga^2\om_\alpha(\ga)^2\norm{\gamma(E_n)}_2^2=\sum_{\ga}
\frac{d_\ga^2\om_\alpha(\ga)^2}{(1+c(\ga))^{2n}}\leq C \sum_{\ga}
\frac{d_\ga^2}{(1+c(\ga))^{2n-\alpha}}
\end{eqnarray}
and by (\ref{series}) the series  is convergent if $ 2n-\alpha>\frac{d(G)}{2} $.
 \end{proof}

\begin{definition}\label{polgr}
\rm  Let $ G $ be a compact Lie group.  A weight $ \om $ on $ \wh G
$ is said to have \textit{polynomial growth}, if $ \om $ is bounded
by $ \om_\al $ for some $ \al\in \R^{>0} $, i.e.\ if $ \om(\pi)\leq C
(1+\ta_S({\pi}))^\al,\pi\in\widehat{G} $ (for some constant $ C>0
$).

If $ G $ is any compact group with a weight $\om  $ on $ \wh G $ then we say
that $ \om $ is of \textit{polynomial growth}  if for every normal subgroup $ N
$ such that $ G/N $ is a Lie group the restriction weight $ \om^N $ has polynomial
growth.
 \end{definition}

\begin{example}\label{dimension} \rm Let $G$ be a connected
compact group and let $\om$ be the
dimension weight, i.e.\ $\om(\pi)=d_\pi$. Then $\om$ is of
polynomial growth. In fact, if $N$ is a normal subgroup such that
$G/N$ is a Lie group, then for $\ga\in\wh{G/N}$ we have
$\om^N(\ga)=\om(\ga\circ q)= d_\ga$. Since by (\ref{series}),
$d_\ga\leq C(1+\norm{\ga}_1)^{(d(G/N)/2)+\ve}$ for some $\ve>0$ and
$C>0$, and we can appeal to (\ref{tas=no}).
\end{example}

\begin{proposition}\label{poly-sub}
A polynomial weight $ \om  $ on $ \wh G $ is nonexponential and hence $
G_\om=G$.
 \end{proposition}
\begin{proof}
We assume first that $ G $ is a Lie group. Then since
$\norm{\si}_1\leq Cn\norm{\pi}_1$ for each $\si\subset\pi^{\otimes
n}$ we have
\begin{eqnarray}
\nonumber \om(\pi^{\otimes n})^{1/n}&=& \sup_{\si\subset\pi^{\otimes
n}}\om(\si)\leq  \sup_{\si\subset\pi^{\otimes
n}}C(1+\norm{\si}_1)^{\alpha/n}\\
\nonumber &\leq& C'(1+n\norm{\pi}_1)^{\alpha/n}\overset{n\to\infty}{\longrightarrow} 1.
\end{eqnarray}
for some constant $ C' $.

Let $ G $ be an arbitrary compact group with polynomial weight on its
dual space $ \wh G $. Take $ \pi\in\wh G $. Then $ G/ \ker \pi $ is
a Lie group.  Let $ N=\ker\pi $ and let $ \pi_N $ be the
representation of $ G/N $ on $\H_\pi$ corresponding to $ \pi $. We
have
\begin{eqnarray}
 \nn \om(\pi^{\otimes n})^{1/n}= \om^N(\pi_N^{\otimes n})^{1/n}.
\end{eqnarray}
By the previous argument, $\lim_{n\to\infty} \om(\pi^{\otimes n})^{1/n}= 1$. That $ G_\om=G
 $ follows from Proposition \ref{polygivesnoth}.
\end{proof}

We observe the following, which was also proved in \cite{parks}.

\begin{corollary}
Let $G$ be a compact group and let $\om$ be the dimension weight.
Then $G_\om=G$.
\end{corollary}


\begin{proof}
Follows from Example~\ref{dimension} and Proposition~\ref{poly-sub}.
\end{proof}
\begin{proposition}
Let $G$ be a compact Lie group, let $\om$ be a symmetric weight such that
$a=\inf_{\gamma\in\wh{G}}\om(\gamma)^{1/\norm{\gamma}_1}>1$. Then $G_\om\ne G$.
\end{proposition}
\begin{proof}
Let $X_1,\cdots,X_n$ and $\OM$ be as in (\ref{eq:casimir}).
Since each $X_i$ is skew-Hermitian we have from (\ref{eq:casimireig}) that
$0\leq -\ga(X_i)^2\leq c(\ga)I_\ga $ for any $
\ga\in\wh G $. Moreover, there exists  $\gamma\in\wh{G}$ such that $
\ga(X_n)\ne 0 $. Set $\theta=\exp{i\lambda X_n}$. 
Then, as in Proposition \ref{prop:GCproperties}, we have $\theta\in
G_{\C}^+$. Since  $ i\pi(X_n)\leq c(\pi)^{1/2}I_{\pi_n} \leq
c\norm{\pi}_1 I_{\pi_n}$ (the last inequality is due to (\ref{equi})),
we have for $\pi\in\what{G}$ that
$$\frac{\noop{\pi(\theta)}}{\om(\pi)} \leq \frac{\noop{\text{exp}{i\lambda\pi(X_n)}}}{a^{\norm{\pi}_1}}\leq
\frac{e^{\lambda c\norm{\pi}_1}}{a^{\norm{\pi}_1}}$$ for each $
\pi\in\wh G $. Taking now $\lambda$ such that $e^{\lambda c}\leq a$
we obtain that $\frac{\noop{\pi(\theta)}}{\om(\pi)}\leq 1$, and
hence $\theta\in G_\om\setminus G$.
 \end{proof}


\medskip\noindent
{\bf A smooth functional calculus and regularity of $A_\ome(G)$.}

\begin{definition}\label{l2om}
\rm
For $ \pi\in\wh{G} $ denote by $ \ch_\pi $ the normalized character of $ \pi $
i.e.
\begin{eqnarray}
 \nn \ch_\pi(s)=d_\pi \text{Tr}(\pi(s)),\; s\in G.
\end{eqnarray}
Then we have for any $ \si\in\wh{G} $
\begin{equation}
\nonumber \si(\ch_\pi)=\left\{\begin{array}{cc}
0 &\text{if }\si\ne\pi\\
I_\pi &\text{if }\si=\pi.\\
\end{array}
\right.
\end{equation}

Let $ Q_\om $ denote the linear operator on Trig$ (G) $ defined by
\begin{eqnarray}
 \nn Q_\om ( u)=\sum_{\wh{G}}\sqrt {\om(\pi)}\ch_\pi\ast u, u\in\text{
Trig}(G),
\end{eqnarray}

and let $ R_\om $ be its inverse:
\begin{eqnarray}
\nonumber R_\om(u)&=&  \sum_{\wh{G}}
\frac{1}{\sqrt{\om(\pi)}}\ch_\pi\ast u,
u\in\text{ Trig}(G).
\end{eqnarray}
Then we can extend the linear operator $ Q_\om $  (resp. $ R_\om $) to an
isometry $ Q_\om:L^2_\om( G)\to L^2(G)$   (resp. $ R_\om:L^2(G)\to L^2_\om(G)$)
and for $ \xi\in L^2(G) $ we have that
\begin{eqnarray}
 \nn  \xi\in L^2_\om(G)\Leftrightarrow Q_\om(\xi)\in L^2(G).
\end{eqnarray}
and
\begin{eqnarray}
 \nn \no\xi_{2,\om}=\no{Q_\om(\xi)}_2.
\end{eqnarray}
We can consider $ Q_\om $ (resp. $ R_\om $) as a convolution operator with the
central  distribution $ q_\om=\sum_{\pi\in\wh{G}}\sqrt{\om(\pi)}\ch_\pi $ (resp. with
the central distribution $ r_\om=\sum_{\pi\in\wh{G}}\frac{1}{\sqrt{\om(\pi)}}\ch_\pi $)
and we shall write $ Q_ \om$ and $ R_\om $ as convolution operators and then
\begin{eqnarray}
 \nn Q_\om\ast u=u\ast Q_\om,\  R_\om\ast u=u\ast R_\om,\ u\in  \text{Trig}(G).
\end{eqnarray}

 \end{definition}

\bigskip
For a real number $a$, let $[a]$ be the entire value of  $a$
and let $d(G)$ denote the dimension of the group $G$.
Let for $ \al>0 $
\begin{eqnarray}\label{rgdef}
  r(G,\al)= \left[\frac{d(G)}{2}+\al\right]+1.
\end{eqnarray}
The following theorem is an adaptation of Theorem 3.1 in
\cite{Lu-Tu}.

\begin {theorem}\label{growthcomp}
Let $G$ be a connected compact Lie group, let $ \om\leq c \om_S^\al $ be a
symmetric polynomial weight on $ \wh{G} $ and let $u=\overline u$
be a self-adjoint element of $A(G)_{\om_S^{r(G,\al)}}$. Then there exists a
positive constant
$C=C(u)$ such that
\begin{equation}\label{growthestimate}
\Vert e^{it u}\Vert_{A_\om(G)}\leq C(1+\val t)^{d(G)/2+\al}, t\in\R.
\end{equation}
\end{theorem}
\begin{proof}  By (\ref{equi}) and (\ref{series}),
we know that \begin{equation}\label{delta} \sum_{\ga\in\widehat G}
\frac{d_\ga^2}{(1+c(\ga))^{s}}<\infty,\ \ \forall s>\frac{d(G)}
2.
\end{equation}
Take $N\in \N^*$ and let  $$\wh{G}_N=\{ \ga\in\wh{G};\ \Vert
\ga\Vert \leq N\}.$$ Take an
$L^2$-function $E_m$ on $G$, such that $$
  \ga(E_m)=\frac{1}{(1+c(\ga))^{m}}I_\ga, \;\forall
\ga\in\widehat
  G,$$
  for all real $m>\frac{d(G)} 4$. We have
\begin{equation}\label{inverse}E_m*(1-\Omega)^m g= g
\end{equation}
for every  $g\in L^2_{\om_{2m}}$ on $G$.
Denote also by $F_N$ the element in $L^2(G)$, for which
$$
  \ga(F_N)=I_\ga,\; \forall \ga\in\widehat
  G_N,\
  \ga(F_N)=0\text{ otherwise},$$
  and by $E_{m,N}$
   the element in $L^2(G)$,
for which
$$
  \ga(E_{m,N})=\ga(E_m), \forall \ga\not\in\widehat
  G_N,\
  \ga(E_{m,N})=0\text{ otherwise}.$$
  Write now  $g_t$ for $e^{itu},\ t\in \R$. We decompose
  $g_t$ into
  $$g_t=a_{t,N}+b_{t,N},$$
  where $a_{t,N}$ and
  $b_{t,N}$ are defined by
$$\ga(a_{t,N})=\ga(g_t), \forall \ga\in\widehat
  G_N,\
  \ga(a_{t,N})=0\text{ otherwise},$$
  $$\ga(b_{t,N})=0, \forall \ga\in\widehat
  G_N,\
  \ga(b_{t,N})=\ga(g_t)\text{ otherwise}.$$
Then $a_{t,N}$ is a $C^\infty$-vector and so $b_{t,N}=g_t-a_{t,N}$ is of
class
$r(G,\al)$.
By the definition of $F_N$, we have that
$$F_N*a_{t,N}=a_{t,N}=F_N\ast Q_\om\ast R_\om\ast a_{t,N}.$$
Hence, by Proposition \ref{prop:ltwofactor},
\begin{eqnarray*}
\Vert a_{t,N}\Vert_{A_\om(G)}&\leq &\Vert F_N\ast Q_\om\Vert_{2,\om}
\Vert R_\om \ast a_{t,N}\Vert_{2,\om}\leq \Vert F_N\Vert_{2,\om^2} \Vert
g_t\Vert_2\\
&\leq&
\Vert F_N\Vert_{2,\om^2} \Vert g_t\Vert_{L^\iy(G)}.
\end{eqnarray*}
 Now, by \cite[proof
of 5.6.7]{wa}, if we set $n=d(G)$ for simplicity, we have
\begin{eqnarray}
\nonumber \Vert F_N\Vert_{2,\om^2}^2&=&\sum_{\Vert\ga\Vert\leq N}\om^2(\ga)
d_\ga^2\\
\nn &\leq& c_3N^{2\al} \sum_{j=0}^N j^{(n-l-r)}(2r+l)(2j+1)^{(r+l-1)}
 \\
\nonumber &\leq&   c_4 N^{2\al}\sum_{j=0}^N j^{n-1}\leq c_5 N^{n+2\al}.
\end{eqnarray}

Hence, since $u$
is a continuous real-valued function,
we have that $ g_t=e^{it
u}\in L^\infty(G), \ t\in\R$,  and  $\Vert g_t\Vert_\infty =1$ and we see
that
$$\Vert a_{t,N}\Vert_{A_\om(G)}\leq C N^{\frac {d(G)} 2+\al}$$
for a certain constant $C>0$. Now for the norm of the element
$b_{t,N}$ we get, using (\ref{inverse}) for $g=b_{t,N}$ and
$m=\frac{1}{2}([\frac{d(G)}{2}]+1)$ (which is easily checked to be
strictly larger than $\frac{d(G)}{4}$),
$$E_m*(1-\Omega)^m b_{t,N}=E_{m,N}*(1-\Omega)^m g_{t}$$
  and so
\begin{eqnarray}
\nonumber \Vert b_{t,N}\Vert_{A_\om(G)}&=& \Vert
E_{m,N}*(1-\Omega)^mg_t\Vert_{A_\om(G)} \\
\nn &=& \Vert Q_\om\ast E_{m,N}*R_\om\ast (1-\Omega)^mg_t\Vert_{A_\om(G)} \\
\nonumber &\leq&  \Vert E_{m,N}\Vert_{2,\om^2}\Vert (1-\Omega)^m g_{t}\Vert_{2}.
\end{eqnarray}

The arguments in \cite[Proof of Theorem 3.1]{Lu-Tu} give the following estimate:

 \begin{equation}\label{ineqfe}
 \nonumber \no{(1-\Omega)^{m} g_t}_2\leq C(1+\val t^{2m}),
 t\in \R,
 \end{equation}

for some  constant $C>0$ and for all $m\in \frac{1}{2}{ \N}$.

Therefore we get the following estimate of the $A_\om(G)$-norm of
$b_{t,N}$: for $t\in\R$,
\begin{eqnarray}
\nonumber \Vert b_{t,N}\Vert_{A_\om(G)}&\leq& C_1\left(\sum_{\Vert \ga\Vert >N}
\frac{d_\ga^2\om(\ga)^2}{(1+c(\ga))^{2m}}\right)^{1/2}(1+\vert t\vert)^{2m} \\
\nonumber &\leq&
C_2\left(\sum_{ j>N} j^{(n-1)} \frac {j^{2\al}}
{(1+j)^{4m}}\right)^{1/2}(1+\vert
t\vert)^{2m}\\
\nonumber &\leq & C_3 \frac{1} {N^{2m-\frac {n+2\al} 2}} (1+\vert t\vert)^{2m}.
\end{eqnarray}
Hence, if we let  $N$ to be the smallest integer $\geq \vert t\vert$
we obtain
\begin{eqnarray}
\nonumber \Vert e^{itu}\Vert_{A_\om(G)}&\leq& \Vert a_{t,N}\Vert_{A_\om(G)}+
\Vert b_{t,N}\Vert_{A_\om(G)}\\
\nonumber &\leq&  \frac{\rm C}{2} (1+\vert t\vert)^{(d(G)/2)+\al}+
\frac{\rm C}{2}(1+\vert t\vert)^{(d(G)/2)+\al}
\end{eqnarray}
 for a new constant $C>0$.
\end{proof}



\begin{theorem}\label{reg}
Let $ G $ be  a connected compact group and let $ \om$  be  a
polynomial weight on $ \wh{G} $. Then $ A_\om(G) $ is  a regular Banach algebra.
 \end{theorem}
\begin{proof}
Let $ E,F $ be two closed subsets of $ G $ such that $ E\cap F=\es
$. We must find an element $ v\in A_\om(G) $, such that $ v=0 $ on $
E $ and $ v=1 $ on $ F $. Since $ G $ is connected, we can find a
normal subgroup $ N $ of $ G $, such that $ G/N $ is  a Lie group
and such that  $ EN\cap FN=\es $. Hence, since $
A_{\om^N}(G/N)\subset A_\om(G) $ we can assume that $ G $ is  a
connected compact Lie group. The algebra Trig$ (G) $ is uniformly dense in $
C(G) $. Hence there exists a trigonometric polynomial $ u $ on $ G
$, such that $ u=\ol u $ and such that $ \val {u(x)}<\frac{1}{10},
x\in E $ and $ {u(y)}>\frac{9}{10}, y\in F  $. We apply now the
functional calculus of $ C^k $ functions to $ u. $ Choose a function
$ \va:\R\to\R $ with compact support of class $ C^{(d(G)/2)+\al+2}
$, vanishing on the interval $ [-\frac{2}{10},\frac{2}{10}]  $ and
taking the value $ 1 $ on the interval $
[\frac{8}{10},\frac{12}{10}]  $. Then the integral
\begin{eqnarray}
 \nn v=\int_\R \hat \va(t)e^{2\pi  i t u}dt
\end{eqnarray}
converges in $ A_\om(G) $ by Theorem~\ref{growthcomp} and $ v $ has the property that
\begin{eqnarray}
 \nn v(x)= \int_\R \hat \va(t)e^{2\pi i t u(x)}dt=\va(u(a))=0, x\in E,
\end{eqnarray}
\begin{eqnarray}
 \nn v(y)= \int_\R \hat \va(t)e^{2\pi i t u(y)}dt=\va(u(y))=1, y\in F.
\end{eqnarray}
 \end{proof}

\section{Spectral synthesis }\label{weak}

Let $A$ be a semisimple, regular, commutative Banach algebra with
$X_A$ as spectrum; for any $a\in A$ we shall denote then by $\hat
a\in C_0(X_A)$ its Gelfand transform. Let also $E\subset X_A$ be a
closed subset. We then denote by
\begin{eqnarray*}
&I_A(E)= \{a \in A\mid \hat{a}^{-1}(0) \text{ contains } E\},\\
&J_A^0(E)=\{a\in A\mid \hat{a}^{-1}(0) \text{ contains a nbhd of
}E\}
\text{ and }
J_A(E)=\overline{J_A^0(E)}.
\end{eqnarray*}

It is known that $I_A(E)$ and $J_A(E)$ are  the largest and the
smallest closed ideals with $E$ as hull, i.e., if $I$ is a closed
ideal such that $\{x\in X_A: f(x)=0 \text{ for all } f\in I\}=E$
then $$ J_A(E)\subset I\subset I_A(E).$$

We say that $E$ is a {\it set of spectral synthesis} for $A$ if
$J_A(E)=I_A(E)$ and of {\it weak synthesis} if the quotient algebra
$I_A(E)/J_A(E)$ is nilpotent (see \cite{warner}).

Let $A^*$ be the dual of $A$. For $a\in A$ we set $\text{supp}(a)=
\overline{\{x\in X_A: \hat a(x)\ne 0\}}$ and
$\text{null}(a)=\{x\in X_A:\hat a(x)=0\}$. For $\tau\in A^*$ and
$a\in A$ define $a\tau$ in $A^*$ by $a\tau(b)=\tau(ab)$ and define
the support of $\tau$ by
$$\text{supp}(\tau)=\{x\in X_A: a\tau\ne 0\text{ whenever } \hat
a(x)\ne0\}.$$
It is known that $\text{supp}(\tau)$ consists of all $x\in X_A$
such that for any neighbourhood $U$ of $x$ there exists $a\in A$
for which $\text{supp}(a)\subset U$ and $\tau(a)\ne 0$. Then, for
a closed set $E\subset X_A$
$$J_A(E)^{\perp}=\{\tau\in A^*:\text{supp}(\tau)\subset E\}$$
and $E$ is spectral for $A$ if and only if
$\tau(a)=0$ for any $a\in A$ and $\tau\in A^*$ such that
$\text{supp}(\tau)\subset E\subset\text{null}(a)$.

We say that $ f\in A $ {\it admits spectral synthesis} if $ f\in J_A(\text{null}(f)) $.

Let $ G $ be a Lie group,  $\omega$ be a symmetric weight on $\wh G$ of
polynomial growth and let $A_\om(G)$ be the corresponding
Beurling-Fourier algebra. Then $ X_{A_\om(G)}=G $ by
Proposition~\ref{poly-sub}, and $A_\ome(G)$ is a semisimple regular commutative Banach algebra of functions on  $G$ by Theorem
\ref{reg} and Theorem~\ref{theo:semisimple}.  In what follows we write $I_\om(E)$ for
$I_{A_\om(G)}(E)$ and $J_\om(E)$ for $J_{A_\om(G)}(E)$. Let $\D(G)$
be the space of smooth functions on $G$.
 For a closed subset $E$ of $G$, we denote by
$J_\D(E)$ the space of all elements of  $\D (G)$ vanishing on $E$.
Note that $\D(G)\subset A_\om (G)$ by virtue of
Proposition~\ref{mdiff}.





\medskip\noindent
{\bf Smooth synthesis.}
\begin{definition}
\rm The  closed subset $E$ of $G$ is said to be of smooth synthesis for $A_\omega(G)$
if $\ol{J_\D(E)}=I_\om(E).$
\end{definition}

The proof of the next theorem is similar to one of \cite[Theorem 4.3]{Lu-Tu}.
\begin{theorem}\label{beurling}
Let $G$ be a Lie group of dimension $n$.
Let $M$ be a smooth submanifold of dimension $m<n$ and let $E$ be a
compact subset of $M$. Let $\om$ be a bounded weight on $\wh{G}$ such that
 $\om\leq C\om_S^\alpha$ for some  $C$, $\alpha>0$.
 Then $\overline{J_\D(E)}^{[\frac m
2+\alpha]+1}=J_\om(E)$.
\end{theorem}

\begin{proof}
As $J_\om (E)$ is the smallest closed ideal whose null set is $E$, in order to prove the statement it is enough
to see that $J_\D(E)^{[\frac m 2+\alpha]+1}\subset J_\om(E)$.
Let $f\in J_\D(E)$.
We note first that denoting by  $\rh(t)$ the right translation by $t$ we have that the mapping
$t\mapsto \rh(t) f\in A_\om(G) $ is $C^\iy$. In fact,
 since for  $m>\frac{d(G)}4+\frac{\alpha}2$, we have $E_m\in L_{\om^2}^2(G)$
and $f=E_m\ast
\check g$, where $g=((1-\OM)^m  f)\check{} \in\D(G)$.  Hence, for
$t\in G$,  we have
that
\begin{eqnarray}
\nonumber \rh(t)f(x)&=&f(xt) \\
\nonumber &=&\int_G E_m(u)g(t\inv x\inv u)du \\
\nonumber &=&\int_G E_m(u)\la(t)g(x\inv u)du\\
\nonumber &=&E_m\ast (\la(t) g)\check {}(x).
\end{eqnarray}
This shows that the mapping $t\mapsto \rh(t)f$ from $G$ to the Banach space
$A_\om(G)$ is smooth.

For $0<\ve<||f||_{\infty}$  let
\begin{eqnarray}{}
\nonumber W_\ve   &=& \{x\in G: \no{\rh (x)f-f}_{A_\om(G)}<\ve\}
\end{eqnarray}
 and
\begin{eqnarray}
\nonumber      \OM_\ve=\{x\in G: \no{\rh(x)f-f}_\iy<\ve\}.
\end{eqnarray}
If $A=\text{inf}_{\pi\in \wh{G}}\om (\pi)$ then by (\ref{agawg})
$$  W_\ve \subset\OM_{\ve/A}.$$

Since  the mapping
$g\mapsto \rh(g) f\in A_\om(G) $ is $C^\infty$,
there exists a
constant $K>0$, an open neighbourhood $W $ of $0$ in the Lie algebra
$\g$ of $G$, such that
\begin{equation}
\nonumber \no {\rh( \exp X)f-f}_{A_\om(G)}\leq K \no X
\end{equation}
for every $X\in \g$ and some fixed norm $\no{\cdot}$ on $\g$. Let
for $\ve>0$, $V_\ve=\exp {B_\ve}$, where $B_{\ve}$ denotes the ball
of radius $\frac{\ve}{2K} $ of center 0 in $\g$. There exist
constants $C_1>C_2>0$ such that  for every $\ve>0$
\begin{equation}
\nonumber C_1\ve^n >\val{V_\ve}>C_2 {\ve^n}
\end{equation}
and $V_\ve\subset W_\ve\subset \OM_{\ve/A}$. In particular, for every
$x=x_0v\in EV_\ve$, $x_0\in E$, $v\in V_{\ve}$, we have that
$f(x_0)=0$ since $f\in J_\D(E)$ and therefore
\begin{eqnarray}{}
\nonumber \val{f(x)}&=&\val{f(x_0v)}  \\
\nonumber &\leq&\val{f(x_0v)-f(x_0)}+\val{f(x_0)}  \\
\nonumber &=&\val{(\rh(v)f-f)(x_0)}\\
\nonumber &\leq&\no{(\rh(v)f-f)}_\iy<\ve A^{-1}.
\end{eqnarray}
Hence
\begin{equation}\label{e^m}
|  f^{[ \frac{m }{2}+\alpha ]+1}(x)|\leq (\ve A^{-1})^{[ \frac{m }{2}+\alpha ]+1}.
\end{equation}

\vspace{0.3cm}


Let $\nu=f^{[\frac m 2+\alpha]+1}$ on $EV_\ve$ and $\nu=0$
elsewhere. Take a  nonnegative function $b(X)\in C^\iy_c(\g)$
supported in $ B_1 $ and let $ b_\ve(X)=b(X/\ve) $, $ X\in\g $. Set
$ u(x)=C b_\ve(\log x) $, where $ C $ is a constant such that
$\int_G u(x)dx=1$. Then $\supp{(u)}\subset V_\ve$. Since
\begin{eqnarray*}
||u||_{2,\om^2}&=&\sum_{\gamma\in\wh{G}}d_\gamma\om(\gamma)^2||\hat u(\gamma)||_2^2
\leq C\sum_{\gamma\in\wh{G}}d_\gamma(1+c(\gamma))^{\alpha}||\hat u(\gamma)||_2^2\\&=&C\sum_{\gamma\in\wh{G}}
d_\gamma||\widehat {((1-\Om)^{\alpha/2}u)}(\gamma)||_2^2
=C||(1-\Om)^{\alpha/2}u||_2^2
\end{eqnarray*}
we have $u\in L^2_{\om^2}(G)$ and $||u||_{2,\om^2}$ behaves like
$\ve^{-n/2-\alpha}$ if $ \alpha/2 $ is an integer, i.e.,
\begin{eqnarray}
 \nn \norm{(1-\Om)^m u}_2^2=\sum_{\pi\in\wh G}(1+c(\pi))^{2m}\norm{\hat
u(\pi)}_2^2 d_\pi \leq C_m\ve^{-n-4m}
\end{eqnarray}
 for some constant $ C_m$  depending on  non-negative integer $ m $.

This gives
\begin{eqnarray}
 \nn \sum_{\pi\in\wh G}(\ve^2(1+c(\pi)))^{2m}\norm{\hat
u(\pi)}_2^2 d_\pi \leq C_m\ve^{-n}.
\end{eqnarray}
Let $ \alpha>0 $ be arbitrary. Then for nonnegative integers $ l $, $ m $ such
that $ 2l\leq\alpha <2m$  we have
\begin{eqnarray}\label{Ã}
 \nn \ve^{2\alpha} \norm{(1-\Om)^{\alpha/2} u}_2^2&=&\sum_{\pi\in\wh
G}(\ve^2(1+c(\pi)))^{\alpha}\norm{\hat
u(\pi)}_2^2 d_\pi \\
\nn  &=&\sum_{\pi:\ve^2(1+c(\pi))<1}(\ve^2(1+c(\pi)))^{\alpha}\norm{\hat
u(\pi)}_2^2 d_\pi\\
\nn  &+&
\sum_{\pi:\ve^2(1+c(\pi))\geq 1}(\ve^2(1+c(\pi)))^{\alpha}\norm{\hat
u(\pi)}_2^2 d_\pi\\
\nn  &\leq&\sum_{\pi\in\wh
G}(\ve^2(1+c(\pi)))^{2l}\norm{\hat
u(\pi)}_2^2 d_\pi\\
\nn  &+&
\sum_{\pi\in\wh
G}(\ve^2(1+c(\pi)))^{2m}\norm{\hat
u(\pi)}_2^2 d_\pi\\
\nn  &\leq&C_l\ve^{-n}+C_m\ve^{-n}\leq C\ve^{-n},
\end{eqnarray}
 and hence
\begin{eqnarray}
\nonumber \norm{(1-\Om)^{\alpha/2} u}_2^2&\leq & C\ve^{-n-2\alpha}
\end{eqnarray}
for some $C>0$.

 Consider now the function
$$\varphi(s)=(f^{[\frac m 2 +\alpha]+1}-\nu)*\check u(s)=\int_G (f^{[\frac m
2 +\alpha]+1}-\nu)(st)u(t)\, dt.$$ By Proposition \ref{prop:ltwofactor},
$\varphi\in A_\om(G)$, and $\va(s)=0$ if $s\cdot\supp{(u)}\subset
EV_\ve$.  As $E\subset\{s:s\cdot\supp{(u)}\subset EV_{\ve}\}$ and
the set $\{s:s\cdot\supp{(u)}\subset EV_{\ve}\}$ is open,
$\supp{(\varphi)}$ is disjoint from $E$ and therefore $\varphi\in
J_\om(E)$. We have
$$f^{[\frac m 2+\alpha]+1}-\varphi=(f^{[\frac m 2+\alpha]+1}-f^{[\frac m
  2+\alpha]+1}*\check u)+\nu*\check u.$$
As $\supp {(u)}\subset V_{\ve}\subset W_\ve$, and $||f^{[\frac m
2+\alpha]+1}-\rh(x)f^{[\frac m
  2+\alpha]+1}||_{A_\om(G)}\leq K\ve$  for all $x\in W_{\ve}$ and some constant
$K=K(m)>0$ which is
   independent of
  $\ve$, it follows that
\begin{eqnarray}{}
  \nonumber &&   ||f^{[\frac m 2+\alpha]+1}-f^{[\frac m
  2+\alpha]+1}*\check u||_{A_\om(G)}    \\
\nonumber &&= ||\int_G (f^{[\frac m 2+\alpha]+1}-\rh(x)f^{[\frac
m
  2+\alpha]+1})u(x)dx||_{A_\om(G)} \\
  \nonumber &&     \leq \int_G||f^{[\frac m 2+\alpha]+1}-\rh(x)f^{[\frac m
  2+\alpha]+1}||_{A_\om(G)}u(x)dx\leq K\ve.
  \end{eqnarray}

We have also $||\nu*\check u||_{A_\om(G)}\leq ||\nu||_2\cdot||u||_{2,\om^2}$. As
   $||u||_{2,\om^2}\leq C|\ve|^{-n/2-\alpha}$ for some constant $ C $, we obtain
\begin{eqnarray*}
\text{dist}(f^{[\frac m 2+\alpha]+1},J_A(E))&\leq& \no{f^{[\frac m
2+\alpha]+1}-\va}_{A_\om(G)}\\
&\leq&K\ve+
C{\ve}^{-n/2-\alpha}\left (\int_{EV_\ve}\val{f^{[\frac m
2+\alpha]+1}(x)}^2dx\right)^{1/2} \\
&\leq &K\ve +C\ve^{-n/2-\alpha}\sup_{x\in EV_\ve}
\val{f^{[\frac m 2+\alpha]+1}(x)}\val {EV_\ve}^{1/2}\\
&\leq &K\ve + C\frac{\ve^{-n/2-\alpha}}{C_2^{1/2}}\ve^{[\frac m 2+\alpha]+1}\val
{EV_\ve}^{1/2}.
\end{eqnarray*}

The following estimation of  $\val {EV_\ve}$ was obtained in \cite{Lu-Tu}: for every small $ \ve >0$
\begin{equation}
\nonumber \val
{EV_\ve}=C \ve^{n-m}.
\end{equation}

Hence, for $\ve>0$ small enough,
\begin{eqnarray*}
 \text{dist}(f^{[\frac{m}{2}+\alpha]+1},J_\omega(E))&\leq& K
\ve+C''\ve^{-n/2-\alpha}\ve^{[\frac m
2+\alpha]+1}\ve^{(n-m)/2}\\
&=&K\ve+C''\ve^{[\frac m 2+\alpha]+1-m/2-\alpha}
\end{eqnarray*}
for a new constant $C''$ which does not depend on $\ve$. Thus
$f^{[\frac{m}{2}+\alpha]+1}\in J_\om(E)$. It follows now from
standard arguments that $J_\D(E)^{[\frac{m}{2}+\alpha]+1}\subset
J_\om(E)$.

\end{proof}

\begin{cor}\label{smoothweak}
Let $E$ be a compact subset of a smooth $m-$dimensional
sub-manifold of
the Lie group  $G$.
If  $E$ is a set of smooth synthesis, then $E$ is of weak
synthesis with  $I_\om(E)^{[m/2+\alpha]+1}=J_\om(E)$.
\end{cor}

The following corollary is a generalization of the Beurling-Pollard theorem for $ A(\Tee) $ (see \cite{kahane}) and
for its weighted analog $ A_{\om}(\Tee^n) $, where $ \om $ is a weight on $ \Zee^n $ defined by $ \om(k)=(1+|k|)^\alpha $ for
$ \alpha>0 $ (see \cite{sht_beurling}).

\begin{corollary} Let $E$ be a compact subset of a smooth $m$-dimensional sub-manifold
of the Lie group  $G$. 
Suppose that $f\in A_\om(G)$ satisfies the
condition
\begin{equation}\label{Lip}
|f(x)|\leq K\inf\{\no{X}: X\in\g, x\exp(-X)\in E\}^r
\end{equation}
for some fixed norm $\no{\cdot}$ on $\g$ and $K>0$.  Then 
$f$ admits spectral
synthesis for $A_\om(G)$ if $r>m/2+\alpha$.
\end{corollary}

\begin{proof}
If $f$ satisfies (\ref{Lip}) then $f$ vanishes on $E$. Let
$V_\ve=\exp{B_\ve}$, where $B_\ve$ is the ball of radius $\ve$ of
center $0$ in $\g$. Let $\nu=f$ on $EV_\ve$ and $\nu=0$ elsewhere
and let $u(x)=u_\ve(x)\in \D(G)$ be family of functions from the
proof of Theorem~\ref{beurling} such that  $u_\ve(x)\geq 0$, $\int_G
u_\ve(x)dx=1$ and $\supp{(u_\ve)}\subset V_\ve$. Using arguments in
the proof of Theorem~\ref{beurling} we can see that
$\varphi_\ve=(f-\nu)\ast\check u_\ve\in J_\om(E)\subset
J_\om(\text{null}(f))$, $f=\varphi_\ve+ (f-f\ast\check
u_\ve)+\nu\ast\check u_\ve$ and
$$ \lim_{\ve\to 0}\norm{f-\varphi_\ve}_{A_\om(G)}= \lim_{\ve\to 0}\norm{\nu\ast\check{u}_\ve}_{A_\om(G)}$$
while
$$  \norm{\nu\ast\check{u}_\ve}_{A_\om(G)}\leq \ve^{-n/2-\alpha}\sup_{x\in EV_\ve}
\val{f(x)}\val {EV_\ve}^{1/2}\leq  C\ve^{-n/2-\alpha}\ve^r\ve^{n/2-m/2}.$$
Hence  if
 $r>m/2+\alpha$, then  $ \lim_{\ve\to 0}\norm{f-\varphi_\ve}_{A_\om(G)}=0  $ and hence
$f\in J_\om (E)\subset J_\om(\text{null}(f))$.
\end{proof}

\begin{theorem}\label{orbit}
Let $ G $ be a Lie group and $ \om $ be a bounded weight on $ \wh G $ such that $\om\leq C\om_S^\alpha$ for
some  $C$, $\alpha>0$.
 Let $ B $ be a group of affine transformations of $ G $ which preserves $ A_\om(G) $, i.e.
$ u(b(x))\in A_\om(G) $ for each $ u(x)\in A_\om(G)  $ and each $ b\in B $. Let $ O\subset G $ be a closed
$ m $-dimensional $ B $-orbit in $ G $. Then $ O $ is a set of smooth synthesis and hence of weak synthesis with
$ I_\om(O)^{[m/2+\alpha]+1}=J_\om(O) $.
 \end{theorem}
\begin{proof}
The proof repeats the arguments of the proof of \cite[Theorem 4.8 and Corollary 4.9]{Lu-Tu}, the affine
transformations of $ G $ (see \cite{Lu-Tu} for the definition) are assumed to preserve the algebra $ A_\om(G) $.
\end{proof}

\begin{corollary}\label{onepoint}
Let $ G $ be a Lie group and $ \om $ be a symmetric weight on $ \wh
G$.

{\bf (i)} If $\om\leq C\om_S^\alpha$ for some  $C$, $0<\alpha<1$ then
each one-point set is a set of spectral synthesis for $ A_\om(G) $.

{\bf (ii)} If $\om\geq C\om_S^\alpha$ for some  $C$, $\alpha\geq 1$ then
no one-point set is a set of spectral synthesis for $ A_\om(G) $
 \end{corollary}
\begin{proof}
Thanks to Proposition \ref{transl},
it is enough to prove the statements for the set $ \{e\} $, where $
e $ is the identity element in $ G $.

(i) We first note that $\{e\}$ is a $0$-dimensional set of smooth synthesis.  Indeed,
$\D(G)$ is dense in $A_\om(G)$ and hence for $u\in I_\om(\{e\})$ there exists a 
sequence $\{u_n\}\subset \D(G)$ which converges to  $u$.  Letting
$u_n'=u_n-u_n(e)$ we have $u_n'\in J_\D(\{e\}$ and
$\norm{u-u_n'}_{A_\om(G)}\leq \norm{u-u_n}_{A_\om(G)}+\norm{u_n(e)-u(e)}_\infty
\leq 2 \norm{u-u_n}_{A_\om(G)}$.  Hence it follows form Theorem \ref{beurling}
that $$I_\om(\{e\})=J_\D(\{e\})=J_\D(\{e\})^{[\alpha+1]}=J_\om(\{e\}).$$

Alternatively, we note that $\{e\}$ is an
orbit under the group $B=\{s\mapsto tst^{-1}\}$ of inner automorphisms, and
we can appeal directly to Theorem \ref{orbit}.

(ii) Assume now $ \alpha\geq 1 $. 
Let $X_1,\dots,X_n$ and $\OM$ be as in (\ref{eq:casimir}).  
Then for $\pi\in\what{G}$ we have by virtue of (\ref{eq:casimireig}), and
the fact that each $X_i$ is skew-hermitian, that 
$$  \frac{\noop{\pi(X_i)}}{\om(\pi)}\leq C\frac{(1+c(\pi))^{1/2}}{(1+\norm{\pi}_1)^\alpha}\leq
C'\frac{(1+\norm{\pi}_1)}{(1+\norm{\pi}_1)^\alpha}\leq C''$$ for some constants
$ C, C',C'' $, and hence $X_i\in A_{\om}(G)^* $.  Thus $\g\subset A_\om(G)$, where each 
element of $\g$ defines a bounded point derivation
at $e$.  We note for each non-zero $X\iin \g$, $I_\om(\{e\})\not\subset \ker X$ 
(indeed $J_\D(\{e\})\not\subset \ker X$), but
$I_\om(\{e\})^2\subset \ker X$.  Hence $\ol{I_\om(\{e\})^2}\subsetneq I_\om(\{e\})$.  \end{proof}

\begin{remark}\rm
(1) Corollary \ref{onepoint} is a generalization of the result about
spectral synthesis of one-point sets for $ A_\om (\Tee^n ) $, where
$ \om $ is the  weight on $ \Zee^n $ defined by
 $\om(k)=(1+\val{k})^{\alpha}  $ (see \cite[Ch.6.3]{reiter}).

(2) For the dimension weight $\ome(\pi)=d_\pi=\ome_S$ ($S=\{\pi_1\}$) on 
$G=\mathrm{SU}(2)$, the failure of spectral synthesis for $A_\om(G)$
at $\{e\}$ was noted in \cite{johnson}.  
\end{remark}



\medskip\noindent{\bf Operator synthesis and spectral synthesis for $A_\om (G)$.}
We let, again $G$ denote a compact group and $\ome$ be a bounded weight.
For a function $u\in A_\om (G)$ and $t$, $s\in G$ define
$$(Nu)(s,t)=u(st^{-1}).$$

Consider the projective tensor product $ L_{\om}^2(G)\hat\otimes L_\om^2(G)$.
Every $\Psi=\sum_if_i\otimes g_i\in  L_{\om}^2(G)\hat\otimes L_\om^2(G)$ can be identified with a function $\Psi:G\times G\to\Cee$ which admits a representation
$$\Psi(s,t)=\sum_{i=1}^\infty f_i(t)g_i(s),$$
 $\sum_i\norm{f_i}_{2,\om}\cdot\norm{g_i}_{2,\om}<\infty$. Such a representation defines a function marginally almost everywhere (m.a.e), i.e. two functions which coincide everywhere apart a marginally null set are identified.
Recall that a subset $E\subset G\times G$ is {\it marginally null} if $E\subset (M\times G)\cup (G\times N)$ and $m(M)=m(G)=0$, where $m$ is the Haar measure on $G$. 
\begin{proposition}
 $Nu\in L_{\om}^2(G)\hat\otimes L_\om^2(G)$.
\end{proposition}

\begin{proof}
Using  the Fourier inversion formula we have
$$Nu(s,t)=\sum_{\pi\in\wh{G}}\text{Tr}(\pi(s)\pi(t^{-1})\hat u(\pi))d_\pi.$$
Let $\{e_i^\pi:i=1,\ldots d_\pi\}$ be an orthonormal basis in $\H_\pi$. Consider for each $\pi\in \wh{G}$ the polar decomposition $\hat u(\pi)=V(\pi)|\hat u(\pi)|$.  Then
\begin{eqnarray*}
Nu(s,t)&=& \sum_{\pi\in\wh{G}}\sum_{i=1}^{d_\pi}(|\hat u(\pi)|^{1/2}\pi(s)e_i^\pi,|\hat u(\pi)|^{1/2}
V(\pi)^*\pi(t)e_i^\pi) d_\pi\\
&=&\sum_{\pi\in\wh{G}}\sum_{i,j=1}^{d_\pi}(|\hat
u(\pi)|^{1/2}\pi(s)e_i^\pi,e_j^\pi)(e_j^\pi,|\hat
u(\pi)|^{1/2}V(\pi)^*\pi(t)e_i^\pi) d_\pi
\end{eqnarray*}
Let 
$$\varphi_{i,j}^\pi(s)=(|\hat
u(\pi)|^{1/2}\pi(s)e_i^\pi,e_j^\pi)d_\pi^{1/2}$$
and
$$\psi_{i,j}^\pi(t)=(e_j^\pi,|\hat
u(\pi)|^{1/2}V(\pi)^*\pi(t)e_i^\pi) d_\pi^{1/2}.$$
In order to show
the statement we have to prove that
$$\sum_{\pi\in\wh{G}}\sum_{i,j=1}^{d_\pi}\norm{\varphi_{i,j}^\pi}_{H}^2<\infty\text { and }
\sum_{\pi\in\wh{G}}\sum_{i,j=1}^{d_\pi}\norm{\varphi_{i,j}^\pi}_{H}^2<\infty.$$
Using the orthogonality property of matrix coefficients one can see that
$$(\hat{\varphi}_{i,j}^\pi(\rho)e_k^\pi,e_l^\pi)=
\frac{1}{d_\pi^{1/2}}(e_k^\pi,|\hat u(\pi)|^{1/2}e_j^\pi)\delta_{il}\delta_{\pi\rho}.$$
Hence
\begin{eqnarray*}
\norm{\hat{\varphi}_{i,j}^\pi(\pi)}_2^2&=&\sum_{k=1}^{d_\pi}\norm{\hat{\varphi}_{i,j}^\pi(\pi)e_k^\pi}^2= \sum_{k,l=1}^{d_\pi}(\hat{\varphi}_{i,j}^\pi(\pi)e_k^\pi, e_l^\pi)(e_l^\pi, \hat{\varphi}_{i,j}^\pi(\pi)e_k^\pi)\\
&=& \sum_{k=1}^{d_\pi}(\hat{\varphi}_{i,j}^\pi(\pi)e_k^\pi, e_i^\pi)(e_i^\pi, \hat{\varphi}_{i,j}^\pi(\pi)e_k^\pi)\\
&=& \frac{1}{d_\pi} \sum_{k=1}^{d_\pi}(e_k^\pi,|\hat u(\pi)|^{1/2}e_j^\pi)(|\hat u(\pi)|^{1/2}e_j^\pi,e_k^\pi)\\
&=& \frac{1}{d_\pi}(|\hat u(\pi)|e_j^\pi,e_j^\pi)
\end{eqnarray*}

and
$$\norm{\varphi_{i,j}^\pi}_{2,\om}^2=\om(\pi)(|\hat u(\pi)|e_j^\pi,e_j^\pi)$$
giving
\begin{eqnarray*}
\sum_{\pi\in\wh{G}}\sum_{i,j=1}^{d_\pi}\norm{\varphi_{i,j}^\pi}_{2,\om}^2&=&\sum_{\pi\in\wh{G}}\sum_{i,j=1}^{d_\pi}\om(\pi)(|\hat u(\pi)|e_j^\pi,e_j^\pi)\\
&=&\sum_\pi d_\pi \om(\pi)\norm{\hat u(\pi)}_1=\norm{ u}_{A_\om (G)}
\end{eqnarray*}

Similar arguments shows that
$$\sum_{\pi\in\wh{G}}\sum_{i,j=1}^{d_\pi}\norm{\varphi_{i,j}^\pi}_{2,\om}^2=\norm{ u}_{A_\om (G)}.$$
The statement is proved.

\end{proof}

Let $H=L_\om^2(G)$. Then the projective tensor product $H\hat\otimes
H$ can be identified with the trace class operators on $H$. The
space $\fL(H)$ of linear bounded operators on $H$ is a dual space of
$H\hat\otimes H$ with the duality given by
$$\langle T,f\otimes g\rangle=(Tf,\bar g).$$

Let $T\in A_\om (G)^*$. Define a linear
bounded operator $S(T)$ on $H$ by
$\widehat{S(T)f}(\pi)=\frac{1}{\om(\pi)} \hat f(\pi)\pi(T)$, where $\bar\pi$ is the conjugate representation.
Since $\sup_{\pi\in\wh{G}}\frac{\noop{\pi(T)}}{\om(\pi)}<\infty$,
$S(T)$ is bounded.

\begin{lemma}\label{nu}
Let $u\in A_\om (G)$, $T\in A_\om (G)^*$. Then $$\langle T,u\rangle=\langle S(T),Nu\rangle.$$
\end{lemma}

\begin{proof}
It is enough to show the statement for a matrix coefficient $u(t)=(\pi(t)e_i^\pi,e_j^\pi)=:c_{ji}^\pi$, where
$\{e_i^\pi,i=1,\ldots d_\pi\}$ is an orthonormal basis in $\H_\pi$.
We have $\langle T,u\rangle=(T_\pi e_i^\pi,e_j^\pi)$.
$$Nu(s,t)=u(st^{-1})=(\pi(t^{-1})e_i^\pi,\pi(s^{-1})e_j^\pi)=\sum_{k=1}^{d_\pi}
(\pi(s)e_k^\pi,e_j^\pi)(e_i^\pi, \pi(t)e_k^\pi).$$
Hence
$$\langle S(T),Nu\rangle=\sum_{k=1}^{d_\pi}(S(T)c_{jk}^{\pi},c_{ik}^{\pi})=
\sum_{k=1}^{d_\pi}\sum_{\rho\in\wh{G}}d_\rho\text{Tr}(\hat{c}_{jk}^{\pi}(\rho)
T_\rho  \hat{c}_{ik}^{\pi}(\rho)^*).$$

Since
$(\hat{c}_{jk}^{\pi}(\rho)e_m^\pi,e_l^\pi)=\frac{1}{d_\pi}
\delta_{\pi\rho}\delta_{kl}\delta_{jm}$ and 
 $$\sum_{k=1}^{d_\pi}\hat{c}_{ik}^{\pi}(\pi)^*\hat{c}_{jk}^{\pi}(\pi)(f)=\frac{1}{d_\pi}(f,e_j^\pi)e_i^\pi$$
we have
\begin{eqnarray*}
\langle S(T),Nu\rangle&=&(T_\pi e_i^\pi,e_j^\pi).
\end{eqnarray*}
\end{proof}

Let $E\subset G$. We set
$$E^*=\{(s,t)\in G\times G: st^{-1}\in E\}.$$
\begin{definition}\rm
If $S\in \fL(L_\om^2(G))$,we define  the support of $S$ (written
$\text{supp}_\fL(S)$) as the set of all points $(s,t)\in G\times G$
with the following property: for any neighbourhoods $U$ of $s$, $V$
of $t$, there are $f$, $g$ in $L_\om^2(G)$ such that
$\text{supp}(f)\subset V$, $\text{supp}(g)\subset U$ and $(Sf,g)\ne
0$.
\end{definition}
\begin{lemma}\label{support}
Let $T\in A_\om(G)^*$. Then $\text{supp}_\fL(S(T))\subset
\text{supp}( T)^*$.
\end{lemma}
\begin{proof}
The proof is similar to one given in \cite[Theorem~4.6]{Ni-Lu}, we include it for the completeness.
It follows from the definition that
$(S(T)f,g)=\langle T,f\ast\bar{\check g}\rangle$ for $f$, $g\in L_\om ^2(G)$, where
$\check g(s)=g(s^{-1})$.

If $(s,t)\in \text{supp}_\fL(S(T))$ but $st^{-1}\notin
\text{supp}(T)^*$, find neighbourhoods $U$ of $s$ and $V$ of $t$
such that $UV^{-1}\cap W=\emptyset$, where $W$ is a neighbourhood of
$\text{supp}(T)$, and then take $f$, $g\in L_\om^2(G)$ supported in
$U$ and $V$ respectively and $(S(T)f,g)\ne 0$. As $(S(T)f,g)=\langle
T, u\rangle$, where $u=  f\ast\bar{\check g}$,
$\text{supp}(u)\subset \overline{UV^{-1}}$, so $u$ vanishes in a
neighbourhood of  $\text{supp}(T)$ and hence $\langle T,u\rangle=0$
giving a contradiction.
\end{proof}

\begin{proposition}\label{sup}
let $S$ be an operator on $L_\om^2(G)$ and let $E$, $F\subset G$ be
closed. Then if $E\times F$ is disjoint from $\text{supp}_\fL(S)$
then $(Sf, g)=0$ for any $f$, $g\in L_\om^2(G)$ supported in $F$ and
$E$ respectively.
 \end{proposition}
\begin{proof}
The proof is similar to \cite[Proposition~2.2.5]{arv}.
\end{proof}

\begin{definition} We say that $E\subset G\times G$ is  a set of operator synthesis with respect to the weight
$\om$ if $\langle S,F\rangle$ for any $S\in \fL(L^2_\om(G))$ and
$F\in L^2_\om(G)\hat\otimes L^2_\om(G)$ with $\text{supp}_\fL
T\subset E$ and $F|_E=0$ m.a.e.
\end{definition}

\begin{proposition}\label{opsynth}
Let $E\subset G$ be closed. If $E^*$ is a set of operator synthesis
with respect to a weight $\om$ on $\wh G$ then $E$ is a set of
spectral synthesis for $A_\om (G)$.
 \end{proposition}

\begin{proof}
Let $T\in A_\om(G)^*$, $u\in A_\om (G)$ such that $\text{supp
}T\subset E\subset \text{null} u$. Then $Nu=0$ on $E^*$ and by
Lemma~\ref{support} $\text{supp}_\fL(S(T)\subset E^*$. The statement
now follows from the equality $\langle T,u\rangle=\langle
S(T),Nu\rangle$ which is due to  Lemma~\ref{nu}.
\end{proof}

\begin{corollary}
Let $G$ be a compact Lie group and let $D=\{(x,x):x\in G\}$. If
$\om$ is a symmetric polynomial weight on $\wh G$ such that $\om\geq
C\om_S^\alpha$ for some $C>0$, $\alpha\geq 1$ then $D$ is not of
operator synthesis with respect to the weight $\om$.

\end{corollary}

\begin{proof}
If $E=\{e\}$ then $E^*=D$. The statement now follows from
Corollary~\ref{onepoint} and Proposition~\ref{opsynth}.
\end{proof}
\vspace{0.5cm}


\medskip

Addresses:
\linebreak
 {\sc
Laboratoire LMAM, UMR 7122, D\'epartement de Math\' ematiques, Universit\' e
Paul Verlaine-Metz, Ile du Saulcy, F-57045 Metz, France.\\
\linebreak Department of Pure Mathematics, University of Waterloo,
Waterloo, ON\quad N2L 3G1, Canada.\\
\linebreak Deprtment of mathematical Sciences, Chalmers University of Technology and the University of Gothenburg, SE-41296 Gothenburg, Sweden.}

\medskip
Email-adresses:
\linebreak
ludwig@univ-metz.fr
\linebreak nspronk@uwaterloo.ca
\linebreak
turowska@chalmers.se

\end{document}